\documentclass[12 pt]{article}
\usepackage{graphicx,mathtools,amsmath,amssymb,amsthm,enumerate, hyperref, xcolor,mathrsfs}
\usepackage[margin=1in]{geometry}
\usepackage{marvosym, fontawesome}
\usepackage[indent = 20 pt]{parskip}

\usepackage{url}
\usepackage{enumitem}

\newtheorem{theorem}{Theorem}[section]
\newtheorem{corollary}[theorem]{Corollary}
\newtheorem{lemma}[theorem]{Lemma}
\newtheorem{proposition}[theorem]{Proposition}

\newtheorem{letterthm}{Theorem}

\theoremstyle{definition}
\newtheorem{definition}[theorem]{Definition}

\newtheorem{construction}[theorem]{Construction}

\theoremstyle{remark}
\newtheorem{remark}[theorem]{Remark}

\newcommand{\cQ}{\mathcal{Q}}
\newcommand{\cN}{\mathcal{N}}
\newcommand{\cK}{\mathcal{K}}
\newcommand{\cJ}{\mathcal{J}}
\newcommand{\cB}{\mathcal{B}}

\newcommand{\cE}{\mathcal{E}}

\newcommand{\cG}{\mathcal{G}}

\newcommand{\cH}{\mathcal{H}}

\newcommand{\cR}{\mathcal{R}}
\newcommand{\cI}{\mathcal{I}}
\newcommand{\cS}{\mathcal{S}}
\newcommand{\cU}{\mathcal{U}}
\newcommand{\cZ}{\mathcal{Z}}
\newcommand{\cP}{\mathcal{P}}
\newcommand{\cV}{\mathcal{V}}

\newcommand{\C}{\mathbb{C}}

\newcommand{\actson}{\curvearrowright}
\newcommand{\innerproduct}[2]{\langle #1, #2 \rangle}

\DeclareMathOperator{\GI}{\mathcal{G}_{\mathcal{I}}}
\DeclareMathOperator{\Idem}{Idem}
\DeclareMathOperator{\phitilde}{\widetilde{\phi}}
\DeclareMathOperator{\dI}{d_{\mathcal{I}}}
\DeclareMathOperator{\id}{id}
\DeclareMathOperator{\Ad}{Ad}
\DeclareMathOperator{\Aut}{Aut}

\DeclareMathOperator{\Out}{Out}

\DeclareMathOperator{\Iso}{Iso}

\DeclareMathOperator{\Isot}{Isot}

\usepackage[doi=false, url =false , isbn = false,style=alphabetic,sorting=nyt, backend = biber, maxcitenames=50, maxalphanames = 5, maxnames=50]{biblatex}
\makeatletter
\def\thanks#1{\protected@xdef\@thanks{\@thanks
        \protect\footnotetext{#1}}}
\makeatother

\begin{document}

\title{Measured inverse semigroups and their actions on von Neumann algebras and equivalence relations}
\author{
Soham Chakraborty \thanks{
\hspace{-2 em} \faMapMarker\hspace{0.18 em}: École Normale Supérieure, Paris, France \\ 
\Letter : \texttt{soham.chakraborty@ens.psl.eu}
}
} 
\date{\today}

\setlength{\parindent}{0em}

\maketitle 

\begin{abstract}\noindent
 It is known to experts that certain regular inclusions of von Neumann algebras arise as crossed products with cocycle actions of the canonical quotient groupoids associated with the inclusions. Similarly, `strongly normal' inclusions of standard equivalence relations arise as semi-direct products with cocycle actions of the quotient groupoids. However, to the author's knowledge, rigorous proofs of these results in full generality are absent in the literature. In this article, we exploit the usual correspondence between inverse semigroups and groupoids, and give a unified approach to proving these `folklore' results and fill this gap in the literature. 
\end{abstract}

\section{Introduction}

Suppose that $G$ is a discrete countable group, $H \mathrel{\unlhd} G$ is a normal subgroup, and let $Q \coloneqq G/H$. Then there are maps $\phi: Q \rightarrow \Aut(H)$ and $c: Q \times Q \rightarrow H$ that satisfy two conditions: 
\begin{enumerate}
    \item $\phi_{q} \phi_{r} = \Ad(c(q,r)) \circ \phi_{qr}$
    \item $c(q,r)c(qr, s)  = \phi_{q}(c(r,s)) c(q,rs)$.
\end{enumerate}
such that there is an isomorphism between the `twisted' semi-direct product $H \rtimes_{(\phi,c)} Q$ and $G$, preserving the subgroup $H$.  We call such a pair $(\phi,c)$ a cocycle action of $Q$ on $H$. If the centralizer $C_G(H)$ is contained in $H$ (and is thus equal to the center $\cZ(H))$, the induced map $Q \rightarrow \Out(H)$ is injective (the cocycle action is outer). One expects the same principle to work in an appropriate sense for `sufficiently normal' inclusions $B \subset M$ of von Neumann algebras. 

A regular subalgebra $B$ of a von Neumann algebra $M$ (with a separable predual) is one whose normalizer generates $M$. Suppose that $B' \cap M = \cZ(B)$ and there is a faithful normal conditional expectation $E: M \rightarrow B$. Then it turns out that there is, in fact such a  `quotient' object, which is a discrete measured groupoid and a 2-cocycle with values in the unitary group $\cU(B)$, such that the corresponding `free' cocycle action on $B$ yields the inclusion $B \subset M$ as the crossed product. This principle has already been used in the literature, for example in \cite{Popa-Shlyakhtenko-Vaes20} and \cite{Chakraborty24a} for classifying regular subalgebras of injective factors.    

The first rigorous treatment of this appears in \cite{Feldman-Moore-1} and \cite{Feldman-Moore-2} where the case where $B$ is abelian is treated. Such subalgebras are called Cartan subalgebras, and in such a situation the quotient groupoid is a countable measured equivalence relation $\cR$. In this case, the 2-cocycle $c$ takes values in the circle and the inclusion $B \subset M$ is isomorphic to the inclusion of $B$ into the twisted equivalence relation von Neumann algebra $L(\cR,c)$. On the other extreme, when $B$ is a factor, then in \cite{Choda79} it was shown that the quotient groupoid is simply the group $G = \cN_M (B)/ \cU(B)$ and  $B \subset M$ arises as the crossed product of a twisted outer $G$-action on $B$. Both Feldman-Moore and Choda's results are of immense importance to the field and have found many applications to date.

The purpose of this article is to provide a rigorous proof for the general case where $B \subset M$ is any regular inclusion with expectation and $B'\cap M = \cZ(B)$. To do this, we use ideas from \cite{DonsigEtAl21} and exploit an intermediate `quotient' object that we call a `measured inverse semigroup'. The set of partial isometries $v \in M$ that have source and range in $\cZ(B)$, normalize $B$ (i.e., $vBv* \subset B $ and $ v*Bv \subset B)$ and satisfy $vBv^* = Bvv^*$ form an inverse semigroup $\cP$. By identifying two such partial isometries with the same source and range, that differ by a partial isometry in $B$, we find that the quotient inverse semigroup $\cI$ has some natural completeness and separability properties (see Definitions \ref{Def: complete semigroups} and \ref{Def: separable inverse semigroups}). Moreover, the set of idempotents of this inverse semigroup forms a lattice that is isomorphic to the projection lattice in $\cZ(B)$. We axiomatize this and introduce the notion of `complete separable measured (csm)' inverse semigroups. In the special case of Cartan triples investigated in \cite{DonsigEtAl21}, these csm inverse semigroups were called `Clifford inverse monoids'.  

In Proposition \ref{Prop: full pseudogroup is an inverse semigroup}, we note that the `full pseudogroup' $[[\cG]]$ of a discrete measured groupoid $\cG$ (the inverse semigroup of partial Borel bisections up to measure zero) is a csm inverse semigroup. Our first main result is the converse to this, which lets us construct a canonical quotient groupoid from a regular inclusion. In Section \ref{Sec: The inverse semigroup - groupoid correspondence} we prove the following result: 
\medskip
\begin{letterthm}
\label{Theorem A}
     Let $\mathfrak{C}_{\text{groupoids}}$ be the category of discrete measured groupoids with groupoid homomorphisms as morphisms and let $\mathfrak{C}_{\text{csm}}$ be the category of csm inverse semigroups with lattice homomorphisms (Definition \ref{Def: Isomorphism of inverse semigroups}) as morphisms. Then the map $\mathfrak{F}: \mathfrak{C}_{\text{groupoids}} \rightarrow \mathfrak{C}_{\text{csm}}$ given by $\mathfrak{F}(\cG) = [[\cG]]$ is a functor that gives an equivalence of categories. 
\end{letterthm}

For topological groupoids, such a correspondence with inverse semigroups is well known, and appears for example in \cite[Chapter 4]{Paterson1998GroupoidsIS}. The discrete measured groupoid associated to the csm inverse semigroup $\cI$ as above is called the `quotient groupoid' associated the inclusion $B \subset M$. In Section \ref{Sec: Cocycle actions and regular subalgebras} we formally define the notion of actions of csm inverse semigroups on von Neumann algebras. The natural action of $\cP$ on $B$ does not lift to an action of $\cI$ and one gets a 2-cocycle and a twisted action of $\cI$ on $B$. Now using Theorem $\ref{Theorem A}$, we can lift this to a cocycle action of the quotient groupoid. The main result of Section \ref{Sec: Cocycle actions and regular subalgebras} is the following generalization of \cite[Theorem 4]{Choda79}.
\newpage
\begin{letterthm}
\label{Theorem B}
    Let $M$ be a von Neumann algebra and $B \subset M$ be a subalgebra with a faithful normal conditional expectation $E: M \rightarrow B$. Then there is a quotient discrete measured groupoid $\cG$ and a free cocycle action $(\alpha,c)$ of $\cG$ on $B$ such that the inclusions $B \subset M$ and $B \subset B \rtimes_{(\alpha,c)} \cG$ are isomorphic and the isomorphism intertwines the conditional expectations if and only if $B$ is regular in $M$ and $B'\cap M = \cZ(B)$.    
\end{letterthm}

Similarly for measured equivalence relations, an analogous notion of `strongly normal' subequivalence relations was introduced in \cite{feld-suth-zimm-88}. Indeed a subequivalence relation $\cS$ of $\cR$ turns out to be strongly normal if and only if $L(\cS)$ is a regular subalgebra of $L(\cR)$. In this setting as well, there is a quotient groupoid $\cG$ acting with a 2-cocycle on $\cS$ such that semi-direct product is $\cR$. In fact the quotient groupoid of the inclusion $\cS \subset \cR$ is isomorphic to the quotient groupoid of $L(\cS) \subset L(\cR)$. Essentially using the same approach as von Neumann algebras and appealing once again to Theorem \ref{Theorem A}, we prove the following result in Section \ref{Sec: quotients of equivalence relations}. 
\medskip
\begin{letterthm}
\label{Theorem C} \textup{(see also \cite[Theorem 2.2]{feld-suth-zimm-89})}
    Let $\cR$ be an ergodic equivalence relation on $(X,\mu)$ and $\cS \subset \cR$ be a subequivalence relation. Then there is a quotient groupoid $\cG$ and a free cocycle action $(\alpha, \Phi)$ of $\cG$ on $\cS$ such that the inclusion $\cS \subset \cR$ is isomorphic to $\cR \subset \cS \rtimes_{(\alpha,\Phi)} \cG$ where $\cS \rtimes_{(\alpha,\Phi)} \cG$ is the semi-direct product equivalence relation if and only if $\cS$ is strongly normal in $\cR$ (see Section \ref{Sec: quotients of equivalence relations} for the relevant definitions).   
\end{letterthm}

A `similar' construction of a quotient groupoid when the inclusion is normal (and not necessarily strongly normal) already appears in \cite[Theorem 2.2]{feld-suth-zimm-89}. When $\cR$ is ergodic, normality and strong normality coincides but in general strong normality is strictly stronger. 

In the proofs of Theorem \ref{Theorem B} and Theorem \ref{Theorem C}, one has to make measurable choices (of unitaries, isomorphisms, partial isometries, etc) many times. Although morally, the Jankov-von Neumann measurable selection theorem always allows one to do this, in practice it is often tedious to achieve this rigorously. Recently, a unified way to look at measured fields of separable structures (Polish groups, Polish spaces, Hilbert spaces, von Neumann algebras, etc) was developed in \cite{WoutersVaes24}. Their results and methods allow us to make a lot of our definitions (cocycle actions, cocycle conjugacy, etc) and proofs in the context of actions of measured groupoids completely rigorous. The results of this article appear in the author's PhD thesis \cite{thesis}. As mentioned in the abstract, the existence of the quotient groupoid from inclusions of von Neumann algebras or equivalence relations was known to experts and has been used in the literature. Therefore, parts of this article can be considered expository. 

\textbf{Acknowledgement:} The results of this article were obtained when the author was a PhD student at KU Leuven. He was supported by the FWO research project G090420N of the Research Foundation Flanders during his PhD. The author would like to thank his PhD supervisor Stefaan Vaes for introducing him to these topics and for many helpful discussions. The author is also indebted to the members of his PhD examination committee: Anna Duwenig, Amine Marrkachi, Gábor Szabó and Marco Zambon for their comments and feedback on his PhD thesis. The author is currently supported by the ERC advanced grant 101141693.

\section{Preliminaries}
\subsection{Von Neumann algebras}

Recall that a von Neumann algebra $M$ is a strong operator topology (SOT)-closed *-subalgebra of $\cB(\cH)$ for a Hilbert space $\cH$. For our purposes, $M$ will always have a separable predual, or equivalently, we will always assume $\cH$ to be separable. A von Neumann algebra $M$ is called \textit{abelian} if any two elements commute, and on the other extreme, is called a factor if $\cZ(M) = \C \cdot 1$. The following is the main class of von Neumann subalgebras we will investigate in this article: 

\begin{definition}
    Let $M$ be a von Neumann algebra and $B \subset M$ be a *-subalgebra. Then $B$ is called a \textit{regular subalgebra} if the normalizer $\cN_M(B)$ generates $M$ as a von Neumann algebra. Such a pair $B \subset M$ is called a regular inclusion. A *-subalgebra $B \subset M$ is said to satisfy the \textit{relative commutant condition} if $B'\cap M = \cZ(B)$.  
\end{definition}

Notice that when $M$ is non-trivial, $\C$ is a regular subalgebra but does not satisfy the relative commutant condition. If $M$ and $N$ are isomorphic von Neumann algebras, we shall denote by $\Iso(M,N)$ the Polish space of all isomorphisms between $M$ and $N$ with respect to the topology of pointwise ultraweak convergence. The group of unitaries of $M$ will be denoted by $\cU(M)$ and forms a Polish group with the SOT topology. 

We shall often deal with Borel and measured fields of von Neumann algebras and other separable structures (Polish groups, Polish spaces, Hilbert spaces, etc) in this article. Since the definitions are standard, we do not include them here. For a unified framework and rigorous definitions we refer the reader to \cite{WoutersVaes24}. In general as in \cite{WoutersVaes24} we call a family of some separable structures $(B_x)_{x \in X}$ over a standard measure space $(X,\mu)$ a \textit{measured field of separable structures} if there is a co-null subset $X_0$ such that $(B_x)_{x \in X_0}$ is a Borel field of the separable structures. The conull subset $X_0$ will sometimes be called a \textit{Borel domain}. We shall often pick measurable sections of such measured fields by applying the Jankov-von Neumann measurable selection theorem (see \cite[Theorem 18.1]{Kec95}).   

Factors are building blocks of von Neumann algebras. A von Neumann algebra $B$ always admits an \textit{ergodic decomposition}: if $(Z,\eta)$ is a standard measure space such that $\cZ(B) = L^\infty(Z,\eta)$, then there is a measured field of factors $(B_z)_{z \in Z}$ such that the direct integral is isomorphic to $B$. We refer the reader to \cite[Chapter X]{TakesakiVolume2} for more details.   
\medskip
\begin{definition}
    A von Neumann subalgebra $A \subset M$ is called a Cartan subalgebra if $A$ if $A$ is a maximal abelian regular subalgebra of $M$ and if there exists a faithful normal conditional expectation $E: M \rightarrow A$.  
\end{definition}

In the next subsection we discuss how Cartan subalgebras always arise from Borel equivalence relations. A von Neumann algebra $M \subset \cB(\cH)$ is called \textit{injective} if there is a faithful conditional expectation $E: \cB(\cH) \rightarrow M$. By the classification theory of Connes and Haagerup (\cite{Connes76}, \cite{Haagerup87}), a von Neumann algebra is injective if and only if it is \textit{approximately finite dimensional (AFD)}, i.e., it is an increasing union of finite dimensional subalgebras. 

\subsection{Measured groupoids and equivalence relations}

A \textit{discrete Borel groupoid} is an algebraic groupoid which is a standard Borel space $\cG$ such that the unit space $\cG^{(0)}$ is a Borel subset and the source map $s: \cG \rightarrow \cG^{(0)}$, the target map $t: \cG \rightarrow \cG^{(0)}$, the inverse map $i: \cG \rightarrow \cG$ and the composition map $\circ: \cG^{(2)} \rightarrow \cG$ are all countable-to-one Borel maps. Here $\cG^{(2)}$ denotes the Borel subset of all composable pairs in $\cG \times \cG$. We shall denote by $\cG^x$ the countable set $\{g \in \cG \; | \; t(g) = x\}$, by $\cG_x$ the countable set $\{g \in \cG \; | \; s(g) = x\}$ and by $\cG^x_y$ the countable set $\cG^x \cap \cG_y$. For all $x$, the countable group $\cG^x_x$ is called the isotropy group at $x$, and $(\cG^x_x)_{x \in X}$ is a measured field of groups, which is called the \textit{isotropy subgroupoid} and will be denoted by $\Isot(\cG)$. If $\Isot(\cG) = \cG^{(0)}$ then $\cG$ is a countable Borel equivalence relation. 

Given a Borel probability measure $\mu$ on $\cG^{(0)}$, one can define the \textit{source} and \textit{target} measures on $\cG$ by: 
\begin{align*}
    \mu_s(E) = \int_{\cG^{(0)}} \#\{s^{-1}(x) \cap E\} \; d\mu(x) \text{ and } \mu_t(E) = \int_{\cG^{(0)}} \#\{t^{-1}(x) \cap E\} \; d\mu(x)
\end{align*}
A discrete Borel groupoid is called \textit{discrete measured} if $\mu_s$ and $\mu_t$ are equivalent measures (have the same null sets). If $\mu_s = \mu_t$, then the groupoid is called \textit{probability measure preserving (pmp)}. Again if $\Isot(\cG)= \cG^{(0)}$, then $\cG$ is a countable measured equivalence relation on the standard probability space $(\cG^{(0)},\mu)$. A Borel subset $B \subset \cG$ is called a \textit{bisection} if $s$ and $t$ are injective on $B$. A bisection is called a \textit{global bisection} if $s$ and $t$ are also surjective onto $\cG^{(0)}$. The group of all global bisections (identified up to measure zero) is called the \textit{full group} of $\cG$ and denoted by $[\cG]$. It forms a Polish group with the metric given by $d(U,V) = \mu_s(U \Delta V)$. The inverse semigroup of all bisections is called the \textit{full pseudogroup} of $\cG$ and denoted by $[[\cG]]$. 

Now suppose $\cR$ is a countable measured equivalence relation on a standard probability space $(X,\mu)$ (this really means that $X$ is the unit space $\cR^{(0)}$ of the groupoid $\cR$). We denote by $\Aut(\cR)$ the Polish group of automorphisms of $\cR$, i,e., all nonsingular automorphisms $\theta$ of $(X,\mu)$ such that $(x,y) \in \cR$ if and only if $(\theta(x),\theta(y)) \in \cR$. 

The theory of Feldman and Moore (\cite{Feldman-Moore-1}, \cite{Feldman-Moore-2}) associates to every such countable measured equivalence relation $\cR$, a von Neumann algebra $L(\cR)$ such that the unit space $L^{\infty}(X,\mu)$ is a Cartan subalgebra. In fact every Cartan pair $A \subset M$ arises as the von Neumann algebra of an equivalence relation with a cocycle twist. We do not go into details here and refer the reader to \cite{Feldman-Moore-2}. This association $\cR \mapsto L(\cR)$ also gives canonical inclusions $\Aut(\cR) \hookrightarrow \Aut(L(\cR))$ and $[\cR] \hookrightarrow \cU(L(\cR))$. An equivalence relation $\cR$ is called \textit{hyperfinite} if $\cR$ is an increasing union of equivalence relations with finite orbits. An equivalence relation $\cR$ is called ergodic if every $\cR$-invariant Borel subset of $X$ is either null or conull. As one would expect, $\cR$ is hyperfinite if and only if $L(\cR)$ is injective, and $\cR$ is ergodic if and only if $L(\cR)$ is a factor. 

Every discrete measured groupoid $\cG$ has an associated countable measured equivalence relation $\cR$ given by $\{(t(g),s(g)) \; | \; g \in \cG\}$ on $(\cG^{(0)}, \mu)$. The groupoid $\cG$ is called ergodic if the associated equivalence relation $\cR$ is ergodic. Just like the equivalence relation case, one can define the von Neumann algebra $L(\cG)$ associated to a discrete measured groupoid $\cG$, which was first constructed by Hahn in \cite{Hahn78}. The groupoid von Neumann algebra $L(\cG)$ is a factor if and only if $\cG$ is ergodic and satisfies the \textit{infinite conjugacy class (icc)} condition (see \cite{ber-chak-don-kim}).  

We also briefly recall the notion of a basis of a groupoid from \cite{ber-chak-don-kim}. A countable subset $\cB = \{B_0,B_1,B_2,...\}$ of $[[\cG]]$ for a discrete measured groupoid $\cG$ is called a \textit{basis} for $\cG$ if the elements of $\cB$ are mutually disjoint up to measure zero and $\bigsqcup_{n} B_n = \cG$ up to measure zero. As a matter of convention we always assume $B_0 = \cG^{(0)}$. A basis is called \textit{symmetric} if $B \in \cB$ if and only if $B^{-1} \in \cB$. As one would expect, the characteristic functions $\{1_B \; | \; B \in \cB\}$ forms a basis for the Hilbert space $L^2(\cG,\mu_s)$, thus justifying the terminology. We refer the reader to \cite{ber-chak-don-kim} for more details.

\subsection{Measured inverse semigroups}
In this section we introduce the notion of measured inverse semigroups. A partially ordered set $X$ with a partial order is called a \textit{meet-semilattice} if it has a meet (or infimum) for any non-empty finite subset. Similarly, such a set is called a \textit{join-semilattice} if it has a join (or supremum) for any non-empty finite subset. Such a set is called a lattice, if it is both a meet-semilattice and a join-semilattice with the same partial order. 

Now let $(X, \mu)$ be a standard probability space and consider the abelian von Neumann algebra $A = L^{\infty}(X,\mu)$. Let $\mathcal{P}(A)$ be the set of all projections in $A$. Notice that $\mathcal{P}(A)$ has the structure of a lattice. Indeed for a countable set of projections $1_{E_{n}}$ with Borel subsets $E_{n}$, we can define their supremum by $1_{E}$ where $E = \bigcup_{n} E_{n}$. Similarly for finitely many projections $1_{E_{1}},...,1_{E_{k}}$, we can define their infimum as $1_{E}$ where $E = \bigcap_{i=1}^{k}E_{i}$. Of course, the intersection can be a null set in which case the infimum is 0. We shall now define the abstract notion of inverse semigroups, where the set of idempotents will resemble the space of projections of an abelian von Neumann algebra. 
\medskip
\begin{definition}
    \label{Def: inverse semigroup}
    An \textit{inverse semigroup} $\mathcal{I}$ is a set together with an associative binary operation, such that every non-zero element $x \in \mathcal{I}$ has a unique inverse $y \in \mathcal{I}$ in the sense that $x = xyx$ and $y = yxy$. In this case we denote $y$ by $x^{-1}$. We shall always assume that there exists the zero element $0 \in \mathcal{I}$. The set of \textit{idempotents} $\Idem(\cI) \subset \mathcal{I}$ is defined to be the set $\{e \in \mathcal{I} \; | \; e^{2} = e\}$.
\end{definition}

Notice that for every idempotent $e \in \Idem(\cI)$, we have $e^{-1} = e$. Moreover, notice that for any $x \in \mathcal{I}$, by uniqueness of inverses, we have that $(x^{-1})^{-1} = x$. In the next lemma, we record some more facts about abstract inverse semigroups. The following lemma is a collection of some easy facts about inverse semigroups, and we omit the proof.  
\medskip 
\begin{lemma}
\label{Lemma: facts about inverse semigroups}
    Let $\mathcal{I}$ be an inverse semigroup with set of idempotents $\Idem(\cI)$. 
    \begin{enumerate} 
        \item For any $x \in \mathcal{I}$, $(x^{-1})^{-1} = x$. In particular, for $e \in \Idem(\cI)$, $e^{-1} = e$ 
        \item The set $\Idem(\cI)$ is commutative and closed under multiplication
        \item For $x,y \in \mathcal{I}$, we have $(xy)^{-1} = y^{-1}x^{-1}$.
        \item If $e \in \Idem(\cI)$ and $x \in \cI$, then $xex^{-1} \in \Idem(\cI)$.
    \end{enumerate}
\end{lemma}
\medskip
Now let $\mathcal{I}$ be an inverse semigroup with the set of idempotents $\Idem(\cI)$. We have the following natural partial order on $\Idem(\cI)$: namely $e \leq f$ iff $ef = e$. The partial order can in fact be extended to the entire semigroup $\mathcal{I}$ by imposing $x \leq y$ iff there exists $e \in \Idem(\cI)$ such that $x = ey$. In this case, writing $f = y^{-1}ey$, we see that $yf = ey = x$, so equivalently $x = yf$ for the idempotent $f$. For a nonempty finite set of elements $e_{1},e_{2},...,e_{n} \in \Idem(\cI)$, we have an infimum with respect to the partial order defined by $e_{1} \bigwedge e_{2} \bigwedge ...\bigwedge e_{n} = e_{1}e_{2}...e_{n}$. This turns $(\Idem(\cI),\bigwedge)$ into a meet-semilattice. If furthermore for any nonempty finite set $e_{1},...,e_{n} \in \Idem(\cI)$, there exists a unique supremum denoted by $u = e_{1}\bigvee e_{2} \bigvee...\bigvee e_{n} $, that satisfies $ue_{i} = e_{i}u =e_{i}$ for all $1 \leq i \leq n$, then the idempotents also form a join-semilattice and in that case we say that $(\Idem(\cI),\bigvee,\bigwedge)$ is the textit{lattice of idempotents} of $\cI$.  

Now we introduce the notion of orthogonality. For two idempotents $e,f \in \Idem(\cI)$, we say that $e \perp f$ if $ef=0$. Similarly for two general elements $x,y \in \mathcal{I}$, we say $x \perp y$ if $xx^{-1} \perp yy^{-1}$ and $x^{-1}x \perp y^{-1}y$. Suppose for every pair $x,y \in \mathcal{I}$ with $x \perp y$, there exists a unique supremum $u \in \mathcal{I}$ satisfying: 
\begin{equation}
    ux^{-1}x = x, \; xx^{-1}u = x, \; uy^{-1}y = y, \; yy^{-1}u = y \nonumber
\end{equation}
then as before, we call the element $u$ the \textit{join} $x \bigvee y$ of the pair. Recall now that from the discussion in the beginning of this subsection that in the space of projections $\mathcal{P}(A)$ of an abelian von Neumann algebra $A = L^{\infty}(X,\mu)$, there is a natural lattice structure with multiplication acting as the meet and projection to the union of sets acting as the join. 

We call a map $\phi$ between two lattices $(\mathcal{E}, \bigwedge, \bigvee)$ and $(\mathcal{F}, \bigwedge', \bigvee')$ a \textit{boolean algebra homomorphism} if $\phi(e \bigwedge f) = \phi(e) \bigwedge' \phi(f)$ and $\phi(e \bigvee f) = \phi(e) \bigvee' \phi(f)$ for any two elements $e,f \in \mathcal{E}$. If $\phi$ is a bijection we call it a \textit{boolean algebra isomorphism}. 
\medskip 
\begin{definition}
    \label{Def: complete semigroups}
    An inverse semigroup $\mathcal{I}$
    will be called \textit{complete} if any sequence $x_{n} \in \mathcal{I}$ of pairwise orthogonal elements has a supremum $u \in \mathcal{I}$ satisfying: 
    \begin{align*}
        ux_{n}^{-1}x_{n} = x_{n}, \text{ and } x_{n}x_{n}^{-1}u = x_{n} \nonumber
    \end{align*}
    for all $n \in \mathbb{N}$. As earlier, we denote the element $u$ by $\bigvee_{n}x_{n}$.
\end{definition}

In the example of the projection lattice $\cP(A)$ of an abelian von Neumann algebra $A = L^{\infty}(X,\mu)$, any sequence of pairwise orthogonal projections $p_{n}$ corresponds to Borel subsets $E_{n}$ that are pairwise disjoint up to measure zero. Taking the Borel subset $E = \bigcup_{n} E_{n}$, the corresponding projection $p$ is the supremum as in Definition \ref{Def: complete semigroups}. Hence such inverse semigroups are always complete, which motivates the following definition. 
\medskip 
\begin{definition}
    \label{Def: measured semigroups}
    A complete inverse semigroup $\mathcal{I}$ is called \textit{measured} if $(\Idem(\cI), \bigvee, \bigwedge)$ is a lattice of idempotents as above and the following are satisfied: 
    \begin{enumerate}
        \item There is a boolean algebra isomorphism $\phi: \Idem(\cI) \rightarrow \mathcal{P}(A)$ for an abelian von Neumann algebra $A$.
        \item For a pairwise orthogonal sequence of idempotents $e_{n} \in \Idem(\cI)$, we have $\phi(\bigvee_{n} e_{n}) = \bigvee_{n} \phi(e_{n})$. 
    \end{enumerate}
\end{definition}

One can check that the von Neumann algebra in Definition \ref{Def: measured semigroups} is indeed well defined. Suppose that $\cP(A_{1})$ is isomorphic to $\cP(A_{2})$ for $A_{i} = L^{\infty}(X_{i},\mu_{i})$ for $i \in \{1,2\}$ and the isomorphism preserves the suprema of orthogonal projections. Then this isomorphism is induced from a non-singular isomorphism of the underlying measure spaces $X_{1}$ and $X_{2}$ (see \cite[Theorem 15.9]{Kec95}). If $\cI$ is measured then there is an idempotent corresponding to the projection $1 \in L^{\infty}(X,\mu)$, and we shall call the idempotent 1. Clearly $1 \cdot f = f \cdot 1 = f$ for any $f \in \Idem(\cI)$. Similarly for an idempotent $e$ we have a unique idempotent $e^{c}$ satisfying $e \bigwedge e^c = 0$ and $e \bigvee e^{c} = 1$.  
\medskip
\begin{definition}
    \label{Def: separable inverse semigroups}
    We say that a subset of a complete measured inverse semigroup $\mathcal{F} \subseteq \mathcal{I}$ a \textit{generating set} if for every $x \in \mathcal{I}$, there exists a countable family of pairwise orthogonal idempotents $e_{n}$ such that $\bigvee_{n}e_{n} = x^{-1}x$, and elements $v_{n} \in \mathcal{F}$ such that $xe_{n} = v_{n}e_{n}$. This in turn implies that $x = \bigvee_{n}xe_{n} = \bigvee_{n}v_{n}e_{n}$. We call a complete measured inverse semigroup \textit{separable} if there is a countable generating set. Without loss of generality we shall always assume that $v_{0} = 1$.  
\end{definition}

In this article we shall only work with complete separable measured (csm) inverse semigroups. We now define the abstract notion of isomorphisms between such inverse semigroups. 
\medskip 
\begin{definition}
    \label{Def: Isomorphism of inverse semigroups}
    An algebraic inverse semigroup homomorphism $\psi: \mathcal{I} \rightarrow \mathcal{J}$ between csm inverse semigroups is called a lattice homomorphism if the restriction $\psi: \Idem(\cI) \rightarrow \Idem(\cJ)$ is a Boolean algebra homomorphism preserving the countable join. If such a homomorphism is bijective, it is called a \textit{lattice isomorphism}. 
\end{definition}

Notice that such an isomorphism immediately has the property that $\psi(1) = 1$ and $\psi(e^{c}) = \psi(e)^{c}$. Moreover for a countable generating set $\{v_{n}\}$ of $\cI$, it can be checked that the set $\{\psi(v_{n})\}$ is a countable generating set for $\cJ$.   
\medskip 
\begin{lemma}
    \label{Lemma: properties of isomorphism of inverse semigroups}
    Let $\phi: \mathcal{I} \rightarrow \mathcal{J}$ be an isomorphism of csm inverse semigroups. Then we have the following: 
    \begin{enumerate}
        \item For elements $x,y \in \mathcal{I}$, if $x \leq y$ then $\phi(x) \leq \phi(y)$.
        \item For elements $x,y \in \mathcal{I}$, if $x \perp y$ then $\phi(x) \perp \phi(y)$. 
        \item If $x_{n}$ is a sequence of pairwise orthogonal elements in $\mathcal{I}$, then $\phi(\bigvee_{n}x_{n}) = \bigvee_{n}\phi(x_{n})$.
    \end{enumerate}
\end{lemma}

\begin{proof}
    For idempotents $e,f \in \mathcal{I}$, this is clear since if $e \leq f$ then $f = e \bigvee (e^{c} \bigwedge f)$ and since $\phi$ is a lattice isomorphism, we have that $\phi(e) \leq \phi(f)$. Now for two general elements $x,y$, we have that $x \leq y$ if there exists an idempotent $e$ such that $x = ye$. Then $\phi(x) = \phi(y)\phi(e)$, thus proving 1. For 2, suppose there exists an idempotent $f$ in $\mathcal{J}$ such that $\phi(x)f = \phi(y)f$, then applying $\phi^{-1}$ on both sides, we have that $x\phi^{-1}(f) = y\phi^{-1}(f)$, contradicting $x \perp y$. Finally, 3 is clear from Definition \ref{Def: complete semigroups} and the fact that $\phi$ is an algebraic isomorphism.    
\end{proof}

\begin{lemma}
\label{Lemma: existence of expectation}
    Let $\mathcal{I}$ be a complete measured inverse semigroup with idempotents $\Idem(\cI)$. For any $x \in \mathcal{I}$, there exists a unique maximal element $E(x) \in \Idem(\cI)$ such that $xE(x) = E(x)$ and $E(x)x = E(x)$. 
\end{lemma}

\begin{proof}
     Firstly notice that if there exists an element $e$ such that $e = xe = ex^{-1}$, then $xex^{-1} = xe = e$ and $ex = xex^{-1}x = xe = e$. Hence $e$ is the maximal element such that $ex =e$ iff $e$ is the maximal element such that $xe =e$. Now suppose that there exists a non-zero $e_{0} \in \Idem(\cI)$ such that $e_{0}x = xe_{0} = e_{0}$, otherwise the unique element is $0$. Let $A = \{e \in \Idem(\cI) \; | \; ex = xe = e \}$, so $A$ is non-empty. Let $A'$ be a totally ordered subset of $A$. Since $\mathcal{I}$ is measured, for any element $e \in \Idem(\cI)$ let $\mu(e)$ denote the trace of the projection corresponding to $e$. Now we consider the set $\mu(A') = \{ \mu(e) \; | \; e \in A' \} \subseteq [0,1]$. Let $c = \sup \{\mu(e) \; | \; e \in A'\}$, and let $e_{n}$ be an increasing sequence in $A'$ such that $\mu(e_{n}) \rightarrow c$. We now let $f = \bigvee_{n}e_{n}$. Clearly for any $e \in A'$, $e \leq f$. Moreover, $e_{n}x = xe_{n} = e$ for all $n$ implies that $fx = xf = f$, so $f \in A$. Now we can apply Zorn's Lemma to the set $A$ to find a unique maximal $E(x)$ in $\Idem(\cI)$ satisfying the conditions of the lemma.  
\end{proof}

Now we define a metric space structure on inverse semigroups. To begin with, we state the following easy facts. \textit{Fact 1}: $E(x^{-1}) = E(x)$. This is immediate by taking inverses of $E(x)x=E(x)$ and $xE(x) = E(x)$. \textit{Fact 2}: For two elements $x,y \in \mathcal{I}$, the element $e = E(x^{-1}y) = E(y^{-1}x)$ is the unique maximal idempotent such that $e \leq (x^{-1}x) \bigvee (y^{-1}y)$ and $xe = ye$. Similarly $f = E(xy^{-1}) = E(yx^{-1})$ is the unique maximal idempotent such that $f \leq (xx^{-1}) \bigvee (yy^{-1})$ and $fx = fy$. This is also an application of Zorn's Lemma similar to the proof of the previous lemma. \textit{Fact 3}: $E(x^{-1}y) \leq x^{-1}x \bigvee y^{-1}y$ and $E(xy^{-1}) \leq xx^{-1} \bigvee yy^{-1}$. This is also easy to check, for example notice that for $f = E(xy^{-1}) = E(yx^{-1})$, we have $fxx^{-1} = fyx^{-1} = f$ and hence $f \leq xx^{-1}$ and similarly $f \leq yy^{-1}$. \textit{Fact 4}: $E(e) = e$ for all $e \in \Idem(\cI)$. This is true as for any $f > e$, we have that $ef = e$ and hence $e$ is the unique maximal element such that $e^{2} = e$.   
\medskip
\begin{definition}
    \label{Def: metric on inverse semigroup}
    The metric $\dI$ on a csm inverse semigroup $\mathcal{I}$ is defined by $\dI(x,y) = \{\mu(x^{-1}x \bigvee y^{-1}y) - \mu(E(x^{-1}y))\} + \{\mu(xx^{-1} \bigvee yy^{-1}) - \mu(E(xy^{-1}))\}$ (where $\mu$ is as in the proof of Lemma \ref{Lemma: existence of expectation}). By Lemma \ref{Lemma: d_I is a metric} below, we know that this is well defined.
\end{definition}
\medskip
\begin{lemma}
    \label{Lemma: d_I is a metric}
    The function $d_{I}$ defined in Definition \ref{Def: metric on inverse semigroup} is a complete separable metric on $\mathcal{I}$. 
\end{lemma}
\begin{proof}
    By Fact 1, it is clear that $\dI(x,y) = \dI(y,x)$. If $x = y$, then by Fact 4, $\dI(x,y) = 0$. If $\dI(x,y) = 0$, then by Fact 3, $xx^{-1} \bigvee yy^{-1} = E(xy^{-1})$ and $x^{-1}x \bigvee y^{-1}y = E(x^{-1}y)$. By Fact 2, $x(x^{-1}x \bigvee y^{-1}y) = y(x^{-1}x \bigvee y^{-1}y)$. But $x^{-1}x \bigvee y^{-1}y \leq 1$ and so $x(x^{-1}x \bigvee y^{-1}y) \leq x$. Also $x(x^{-1}x) \leq x(x^{-1}x \bigvee y^{-1}y)$ which implies $x \leq x(x^{-1}x \bigvee y^{-1}y)$. Therefore we have that $x = x(x^{-1}x \bigvee y^{-1}y)$ and running the same argument for $y$ instead of $x$, we have that $x = y$. The only thing left to show is the triangle inequality. 
    
    For two elements $x,y \in \cI$, consider the function $d_{0}$ given by $d_{0}(x,y) = \inf \{\mu(e) \; | \; e \in \Idem(\cI) \text{ and } xe^{c} = ye^{c}\}$. Then one can check that $d_{\cI}(x,y) = d_{0}(x,y)+ d_{0}(x^{-1},y^{-1})$. Therefore it suffices to check that $d_{0}$ satisfies the triangle inequality. Now for arbitrary elements $x,y,z \in \cI$, suppose there are idempotents $e$ and $f$ such that $xe^{c}=ye^{c}$ and $yf^{c}=zf^{c}$. By taking $g = e \bigvee f$ we have that $xg^{c}=yg^{c}$ and $yg^{c}=zg^{c}$. Thus $xg^{c}=zg^{c}$ and we have that
    \begin{align*}
        d_{0}(x,z) \leq \mu(g) \leq \mu(e) + \mu(f)
    \end{align*}
    Now by taking the infimum over all such idempotents $e$ and $f$, we have $d_0(x,z) \leq d_0(x,y) + d_0(y,z)$, as required. The proofs of completeness and separability are routine and are left to the reader.  
\end{proof}

\section{Correspondence with measured groupoids}
\label{Sec: The inverse semigroup - groupoid correspondence}

In this section we prove Theorem \ref{Theorem A} and show that discrete measured groupoids and csm inverse semigroups are essentially the same category. The proof of Theorem \ref{Theorem A} will follow from Propositions \ref{Prop: full pseudogroup is an inverse semigroup}, \ref{Prop: from inverse semigroup to a discrete measured groupoid} and \ref{Prop: Groupoid to inverse semigroup is injective}. In this section, we shall often denote $\Idem(\cI)$ by the letter $\cE$ for convenience of notation.  

\medskip
\begin{proposition}
    \label{Prop: full pseudogroup is an inverse semigroup}
    The full pseudogroup $[[\mathcal{G}]]$ of a discrete measured groupoid $\mathcal{G}$ with unit space $(X,\mu)$ is a csm inverse semigroup with $\Idem([[\cG]])$ isomorphic to $\cP(L^{\infty}(X,\mu))$. 
\end{proposition}
\begin{proof}
    Any bisection $U \in [[\cG]]$ has a unique inverse $U^{-1}$ and it is a straightforward calculation to check that $UU^{-1}U = U$ and $U^{-1}UU^{-1} = U^{-1}$, hence $[[\mathcal{G}]]$ has an inverse semigroup structure. The set of idempotents $\mathcal{E}$ is indeed given by $\{U \in [[\mathcal{G}]] \; | \; U \subseteq \mathcal{G}^{(0)}\}$, hence $\mathcal{E}$ is isomorphic to the lattice of projections in $L^{\infty}(X,\mu)$ and $[[\mathcal{G}]]$ is measured. Now if a sequence of bisections $U_{n}$ is pairwise orthogonal, then $\mu(s(U_{i}) \cap s(U_{j})) = 0$ and $\mu(t(U_{i}) \cap t(U_{j})) = 0$ for all $i \neq j$. Then letting $U = \bigcup_{n} U_{n}$, we have that up to a null set, the source and target maps restricted to $U$ are still injective. For such a sequence, $U$ is the unique supremum satisfying $UU_{n}^{-1}U_{n} = U_{n}$ and $U_{n}U_{n}^{-1}U = U_{n}$ for all $n$, hence showing that $[[\mathcal{G}]]$ is complete. 
    
    By \cite[Proposition 3.1]{ber-chak-don-kim} the groupoid $\mathcal{G}$ has a basis $\mathcal{U} = \{U_{n}\}$. Now given any bisection $U \in [[\mathcal{G}]]$, consider the bisections $U_{n} \cap U$. Letting $E_{n} = s(U_{b} \cap U)$, we have that the elements $\{E_{n}\}$ are pairwise orthogonal because if $x \in E_{n} \cap E_{m}$, then there exists $g,h \in U$ with $s(g) = s(h) = x$, which forces $n = m$. We also have that $\bigcup_{n}E_{n} = U^{-1}U = s(U)$. Now we check that $UE_{n} = U_{n}E_{n}$. Indeed if $g \in U_{n}E_{n}$ with $s(g) = x$, then $x \in E_{n}$. This means there exists $h \in U_{n} \cap U$ with $s(h) = x$, which in turn implies that $h = g$ and $g \in UE_{n}$. Similarly if $g \in UE_{n}$, it also has to be in $U_{n}E_{n}$, thus proving that $[[\mathcal{G}]]$ is separable.      
\end{proof}

Before proving the converse to Proposition \ref{Prop: full pseudogroup is an inverse semigroup}, we develop some notation.  For an idempotent $e$, we denote by $\mathcal{E}_{e}$ the set $\{f \in \mathcal{E}  \; | \; f \leq e \}$. For measured inverse semigroups, we know that $\mathcal{E}$ is isomorphic as a lattice to the projection lattice of $L^{\infty}(X,\mu)$. If we denote the Borel subset corresponding to $e$ by $E$, then $\mathcal{E}_{e}$ is isomorphic as a lattice to the projection lattice in $L^{\infty}(E)$. The set $\mathcal{E}_{e}$ forms a Boolean algebra with respect to the operations $\bigvee$, $\bigwedge$ and inverse and it is isomorphic to the $\sigma$-algebra of Borel sets $B(E)$. 

\medskip 
\begin{proposition}
    \label{Prop: from inverse semigroup to a discrete measured groupoid}
    Let $\mathcal{I}$ be a csm inverse semigroup with $\Idem(\cI)$ isomorphic to $\cP(L^{\infty}(X,\mu))$. Then there is a discrete measured groupoid $\GI$ with unit space $(X,\mu)$ such that $[[\GI]]$ is isomorphic to $\mathcal{I}$. 
\end{proposition}

\begin{proof}
    We know that the set of idempotents $\mathcal{E}$ is isomorphic as a lattice to projections in $L^{\infty}(X,\mu)$. We shall construct a groupoid with unit space $(X,\mu)$. For $v \in \mathcal{I}$, let us denote the elements $v^{-1}v$ and $vv^{-1}$ by $e,f \in \mathcal{E}$ respectively. Let $\psi : \mathcal{E}_{f} \rightarrow \mathcal{E}_{e}$ be the lattice isomorphism defined by $\psi(e') = v^{-1}e'v$. Now suppose under the identification with projections in $L^{\infty}(X, \mu)$, we have that $e \leftrightarrow U$ and $f \leftrightarrow V$ for Borel sets $U$ and $V$. Then $\psi$ induces a Boolean algebra isomorphism $B(V) \rightarrow B(U)$ and by \cite[Theorem 15.9]{Kec95} we have a non-singular isomorphism $\phi_{v}: U \rightarrow V$ implementing it.   

    Now we start with a countable generating set $\mathcal{F}_{0} = \{v_{0}=1, v_{1},v_{2}... \} \subset \mathcal{I}$. We consider the countable set $\mathcal{F} = \{w_{0} =1,w_{1},w_{2},...\}$ where $w_{1} = v_{1}E(v_{1})^{c}$ (see Lemma \ref{Lemma: existence of expectation}) and $w_{n}$ is defined as follows. Let $e$ be the unique maximal element such that there exists pairwise orthogonal idempotents $e_{1},e_{2},...,e_{n-1}$ with $v_{n}e = \bigvee_{i=1}^{n-1}w_{i}e_{i}$. We then define $w_{n} = v_{n}e^{c}$. It can be checked that $\mathcal{F}$ is still a generating set, with the property that for any two elements $w_{m}$ and $w_{n}$, there exists no non-zero idempotent $e$ such that $w_{m}e = w_{n}e$. This in turn has the consequence that any element $x \in \mathcal{I}$ can be written uniquely as a join of elements $\bigvee w_{n}e_{n}$. Under the identification with $L^{\infty}(X, \mu)$, suppose we have Borel sets $w_{n}^{-1}w_{n} \leftrightarrow D_{n}$ and $w_{n}w_{n}^{-1} \leftrightarrow R_{n}$ (Note that we take $w_{0} = 1$ and hence $D_{0} = R_{0} = X$). By the previous paragraph, this induces non-singular isomorphisms $\phi_{n} : D_{n} \rightarrow R_{n}$ between Borel sets. 

    We now construct our groupoid $\mathcal{G}_\cI$ as follows: Let $D_{0}' = D_{0} = X$ and we inductively define $D_{n}'= \{x \in D_{n} \; | \; \phi_{i}(x) \neq \phi_{n}(x) \text{ for all } i < n\}$. By the properties of $\mathcal{F}$, the set $D_{n}'$ differs from $D_{n}$ only by a measure zero set. For each $n \geq 1$, let $\mathcal{G}_{n}$ consist of elements $g = (\phi_{n}(x),x)$ for all $x \in D_{n}'$ with $s(g) = x$ and $t(g) = \phi_{n}(x)$. Now we define $\GI = \bigsqcup_{n \geq 0} \mathcal{G}_{n}$. By abusing notation we will sometimes denote elements in $\mathcal{G}_{n}$ by $g_{n}$ (with a subscript $n$). We shall now define the composition and inverse operations as follows. Let $h = (\phi_{n}(x),x) \in \mathcal{G}_{n}$ and let $y = \phi_{n}(x)$. Now let $g = (\phi_{m}(y),y) \in \mathcal{G}_{m}$ and we need to define a product $g \cdot h$. We first write $w_{m}w_{n} = \bigvee_{k}w_{k}e_{k}$. This implies that $\phi_{m}\circ \phi_{n}|_{e_{k}} = \phi_{k}|_{e_{k}}$ for every $k$. Let $e_{k}$ correspond to the Borel sets $E_{k} \subset X$ and let $k_{0}$ be such that $x \in E_{k_{0}}$. Then $l = (\phi_{k_{0}}(x),x)$ is the unique element $g \cdot h$ with $s(l) = x$ and $t(l) = \phi_{k_{0}}(x) = \phi_{m}\circ \phi_{n}(x)$. Similarly for an element $g = (\phi_{m}(x),x) \in \mathcal{G}_{m}$ with $\phi_{m}(x) = y$, we first write $w_{m}^{-1} = \bigvee_{k}w_{k}e_{k}$. Hence $\phi_{m}^{-1}|_{e_{k}} = \phi_{k}|_{e_{k}}$ for every $k$. Let $k_{0}$ be such that $y \in E_{k_{0}}$, then $g^{-1} \coloneqq (\phi_{k_{0}}(y),y)$ is the unique inverse of $g$. Indeed we have that $s(g^{-1}) = y$, $t(g^{-1}) = \phi_{k_{0}}(y) = \phi_{m}^{-1}(y) = x$.
    
    Now notice that $\mathcal{G}^{(0)}$ is a standard probability space as the inverse semigroup is measured, and the Borel structure of each $\mathcal{G}_{n}$ is induced from the Borel structure of $\mathcal{G}^{(0)} \times \mathcal{G}^{(0)}$. Since $\mathcal{G}$ is a disjoint union of countably many Borel sets, we have that $\mathcal{G}$ is a standard Borel space as well. Since each $\phi_{n}$ is Borel, the source, target, inverse and composition maps are Borel. It is also clear from construction that the source and target maps are countable to one. Now let $U$ be a Borel subset of $\mathcal{G}$ such that $s(U)$ and $t(U)$ are injective. We write $U = \bigsqcup_{n} U_{n}$ where $U_{n} =  U \cap \mathcal{G}_{n}$. Hence $\mu(s(U)) = \sum_{n} \mu(s(U_{n}))$. Similarly $\mu(t(U)) = \sum_{n}\mu(t(U_{n}))$, but $t(U_{n}) = \phi_{n}(s(U_{n}))$ and hence $\mu(s(U_{n})) = 0$ if and only if $\mu(t(U_{n})) = 0$ for every $n$. Thus the measure is quasi-invariant and $\mathcal{G}$ is a discrete measured groupoid.

    We now show that the csm inverse semigroup $[[\GI]]$ is isomorphic to $\mathcal{I}$ by defining a map $\gamma : \mathcal{I} \rightarrow [[\GI]]$. Once again we use the isomorphism between $\mathcal{E}$ and $\cP(L^{\infty}(X,\mu))$. For an idempotent $e \in \mathcal{E}$ corresponding to the Borel set $e \leftrightarrow E$, we define $\gamma(e) = E$ where $E$ is the corresponding bisection in the unit space $\GI^{(0)}$. Next we define $\gamma (w_{n}) = \mathcal{G}_{n}$ for all $n$ and $\gamma (w_{n}e) = \gamma(w_{n})\gamma(e) = \mathcal{G}_{n} \cap E$. Since any element $u$ can be written uniquely as $\bigvee w_{n}e_{n}$, we can now define $\gamma(u) = \bigsqcup\gamma(w_{n}e_{n})$. By definition and the uniqueness of this expression of $u$, it is immediate that $\gamma$ is injective. Given any bisection $U \in [[\GI]]$, we write $U = \bigsqcup U_{n}$ where $U_{n} \subset \mathcal{G}_{n}$. Let $f_{n}$ be the idempotents corresponding to $s(U_{n})$ and notice that $\gamma(\bigvee w_{n}f_{n}) = U$ and hence $\gamma$ is surjective. Now for idempotents $e$ and $f$, consider the decomposition $w_{m}ew_{n}f = \bigvee w_{k}e_{k}$. Let us assume $e_{n}$ corresponds to the Borel subset $E_{n}$ and $e,f$ corresponds to the Borel subsets $E,F$ respectively. Then we have: 
    \begin{align*}
     \gamma(w_{m}e)\gamma(w_{n}f) &= (E \cap \cG_{m})\cdot (F \cap \mathcal{G}_{n}) = \bigsqcup_{k} (E_{k} \cap \mathcal{G}_{k}) \\ &= \bigsqcup_{k} \gamma (w_{k}e_{k}) = \gamma (\bigvee w_{k}e_{k}) = \gamma(w_{m}ew_{n}f)   
    \end{align*}
    
    This implies that for any two elements $u = \bigvee_{i} w_{i}e_{i}$ and $v = \bigvee_{j} w_{j}f_{j}$, we have:
\begin{align*}
    \gamma(uv) &= \gamma((\bigvee_{i} w_{i}e_{i})(\bigvee_{j}w_{j}f_{j})) = \gamma (\bigvee_{i,j}( w_{i}e_{i}w_{j}f_{j}))\\  &= \bigsqcup_{i,j}\gamma (w_{i}e_{i}w_{j}e_{j}) 
    = \bigsqcup_{i,j}\gamma(w_{i}e_{i})\gamma(w_{j}e_{j}) \\ &= (\bigsqcup_{i} \gamma(w_{i}e_{i}))(\bigsqcup_{j}\gamma(w_{j}e_{j})) = \gamma(u)\gamma(v) 
\end{align*}
Similarly for the inverse, note that: 
\begin{align*}
 \gamma(w_{m}e)^{-1} &= (\mathcal{G}_{m} \cap E)^{-1} = \bigsqcup_{k}(E_{k} \cap \mathcal{G}_{k}) \\ &= \bigsqcup_{k} \gamma (w_{k}e_{k}) = \gamma((w_{m}e)^{-1})   
\end{align*}
Therefore for $u = \bigvee_{k}w_{k}e_{k}$ we have:
\begin{align*}
    \gamma(u^{-1}) = \gamma(\bigvee_{k}(w_{k}e_{k})^{-1}) = \bigsqcup_{k}\gamma(w_{k}e_{k})^{-1} = (\bigsqcup_{k}\gamma(w_{k}e_{k}))^{-1} = \gamma(u)^{-1}. \nonumber
\end{align*}
Hence $\gamma$ is an inverse semigroup isomorphism between $\mathcal{I}$ and $[[\GI]]$. By Proposition \ref{Prop: Groupoid to inverse semigroup is injective} below, we have that the construction does not depend on the choice of the generating set, thus concluding the proof. 
\end{proof}

One checks that the map $\cG \mapsto [[\cG]]$ as in Theorem \ref{Theorem A} is indeed functorial. Notice that for an isomorphism $\phi: \cG_{1} \rightarrow \cG_{2}$ of discrete measured groupoids, we get an isomorphism of inverse semigroups $\phitilde: [[\cG_{1}]] \rightarrow [[\cG_{2}]]$ simply by mapping a bisection $B \subset \cG_{1}$ to $\phi(B) \subset \cG_{2}$. In the next lemma, we show that any such isomorphism of inverse semigroups must be implemented by an isomorphism of groupoids, thus concluding the proof of Theorem \ref{Theorem A}.  

\medskip
\begin{proposition}
\label{Prop: Groupoid to inverse semigroup is injective}
    Let $\cG_{1}$ and $\cG_{2}$ be two discrete measured groupoids such that $\Phi: [[\cG_{1}]] \rightarrow [[\cG_{2}]]$ is an isomorphism of csm inverse semigroups. Then there is an isomorphism of discrete measured groupoids $\phi: \cG_{1} \rightarrow \cG_{2}$ such that $\phitilde = \Phi$. 
\end{proposition}
\begin{proof}
    Let $X$ and $Y$ be the unit spaces of $\cG_{1}$ and $\cG_{2}$ respectively and let $\{X,B_{1},B_{2},...\}$ be a basis for $\cG_{1}$. Since $\Phi$ is an isomorphism of measured inverse semigroups, we have that $\Phi(1) = 1$ and $\Phi$ preserves the lattice structure of the idempotents. Again \cite[Theorem 15.9]{Kec95} immediately gives an isomorphism $\phi_{0}: X \rightarrow Y$ that implements the isomorphism between the lattices $L^{\infty}(X)$ and $L^{\infty}(Y)$. We first claim that $\{Y,\Phi(B_{1}),\Phi(B_{2}),...\}$ forms a basis for $\cG_{2}$. We see again that for bisections $C,D \in [[\cG_{1}]]$, if $C \cap D$ is measure zero, then $\Phi(C) \cap \Phi(D)$ is measure zero. Now, given any bisection $D \in [[\cG_{2}]]$, we have that $\Phi^{-1}(D) = \bigsqcup_{n}B_{n}E_{n}$ for Borel subsets $E_{n} \subseteq X$. Since $\Phi(E_{n})$ are disjoint projections in $Y$, we have automatically that $D = \bigsqcup_{n} \Phi(B_{n})\Phi(E_{n})$. Now since $\cG_{2}$ has a basis, up to measure zero any point belongs to some bisection. Thus we have that $\bigsqcup_{n}\Phi(B_{n}) = \cG_{2}$ up to measure zero and this proves our claim.
    
    Now since each bisection $B$ has the property that the source and target maps are injective on $B$, there is a lattice structure on $B$ induced from $s(B)$. By Lemma \ref{Lemma: properties of isomorphism of inverse semigroups}, we have that $\Phi$ gives a Boolean algebra isomorphism between $B$ and $\Phi(B)$. Once again we can find a Borel isomorphism $\phi_{n}: B_{n} \rightarrow \Phi(B_{n})$. Now we can define $\phi$ as the disjoint union of $\phi_{n}$'s and since this is a countable union of Borel maps, $\phi$ is a Borel isomorphism and $\phitilde = \Phi$ by construction. 
\end{proof}

\section{Quotient groupoid of regular inclusions}
\label{Sec: Cocycle actions and regular subalgebras}

For the entirety of this section, we refer the reader to \cite{WoutersVaes24} for relevant definitions. In particular, we shall use the notions of measured fields of Polish spaces (\cite[Definition 2.4]{WoutersVaes24}), Polish groups (\cite[Definition 7.3]{WoutersVaes24}), Hilbert spaces (\cite[Definition 8.2]{WoutersVaes24}), von Neumann algebras (\cite[Definition 8.4]{WoutersVaes24}). The main goal of this section is to prove Theorem \ref{Theorem B}, i.e., there is a one-one correspondence between regular subalgebras of von Neumann algebras and cocycle actions of discrete measured groupoids on measured fields of factors. As observed in \cite[Remark A.2.28]{Lisethesis} there are multiple ways of defining actions of groupoids on measured fields of factors, and they are equivalent. So we give here the following definition:  inspired by \cite[Theorem A.2.25]{Lisethesis}. For a discrete measured groupoid $\mathcal{G}$ with unit space $(X,\mu)$, we shall denote its restriction to a Borel subset $E \subset X$ by $\mathcal{G}|_{E}$. Before defining the central objects that we deal with in this section, we make the following crucial remark: 
\medskip
\begin{remark}
    \label{Rmk: justifying borel domain for groupoids}
    From \cite[Proposition 7.8]{WoutersVaes24} and \cite[Theorem A.2.25]{Lisethesis} we know that for measured fields of factors $B = (B_{x})_{x \in X}$ and $D = (D_{x})_{x \in X}$, the unitary groups $\cU(B) = (\cU(B_{x}))_{x \in X}$ forms a measured field of Polish groups and the *-isomorphisms $\Iso(B,D) = (\Iso(B_{x},D_{x}))_{x \in X}$ forms a measured field of Polish spaces. Now for a discrete measured groupoid $\cG$ with unit space $X$, we know that $\cG$ and $\cG^{(2)}$ are standard Borel spaces. Since the source and target maps are countable-to-one, it is easy to check that all families of factors of the form $(B_{s(g)})_{g \in \cG}$, $(B_{t(g)})_{g \in \cG}$, $(B_{t(g)})_{(g,h) \in \cG^{(2)}}$, etc are measured fields of factors. Hence we get measured fields of Polish groups of the form $(\cU(B_{t(g)}))_{g \in \cG}$, $(\cU(B_{s(g)}))_{g \in \cG}$, $(\cU(B_{t(g)}))_{(g,h) \in \cG^{(2)}}$. In the same way, we also get measured fields of Polish spaces of the form $(\Iso(B_{s(g)},B_{t(g)}))_{g \in \cG}$. We shall often talk about common Borel domains of measured fields of separable structures where the base space varies over the set $\textit{dom} = \{ X, \cG , \cG^{(2)}\}$. Then we shall say that a finite set $\mathscr{S}$ of measured fields of separable structures over the standard Borel spaces in the set $\textit{dom}$ is said to have a common Borel domain $X_{0} \subset X$ if for all $S \in \mathscr{S}$, the set $X_{0}$, $\cG|_{X_{0}}$ or $\cG|_{X_{0}}^{(2)}$ forms a Borel domain for $S$ when $S$ is a measured field over the standard Borel space $X$, $\cG$ or $\cG^{(2)}$ respectively.   
\end{remark}
\medskip
\begin{definition}
    \label{2-cocycle on groupoids}
    Let $\mathcal{G}$ be a discrete measured groupoid with unit space $(X, \mu)$ and $B = (B_{x})_{x \in X}$ be a measured field of factors. Consider the measured field of Polish spaces $\Iso(B_{s(g)},B_{t(g)})_{g \in \mathcal{G}}$. Suppose that $\alpha$ is a section of $(\Iso(B_{s(g)},B_{t(g)}))_{g \in \cG}$.  Then $\alpha$ is called \textit{an action of $\mathcal{G}$ on $B$} if there is a Borel domain $X_{0} \subset X$ such that letting $\cG_{0} = \cG|_{X_{0}}$, we have that $\alpha|_{\cG_{0}}$ is a Borel section that satisfies $\alpha_{g} \circ \alpha_{h} = \alpha_{gh}$ for all $(g,h) \in \mathcal{G}_{0}^{(2)}$. 
    
    A measurable section $u$ of $(\cU(B_{t(g)}))_{(g,h) \in \cG^{(2)}}$ will be called a \textit{measured field of unitaries}. Suppose that $\alpha$ is a measurable section of the measured field of Polish spaces $(\Iso(B_{s(g)},B_{t(g)}))_{g \in \cG}$ and $u$ is a measured field of unitaries in $(\cU(B_{t(g)}))_{(g,h) \in \cG^{(2)}}$. Suppose that $X_{0}$ is a common Borel domain for $(\Iso(B_{s(g)},B_{t(g)}))_{g \in \cG}$ and $(\cU(B_{t(g)}))_{(g,h) \in \cG^{(2)}}$ such that $\alpha|_{\cG_{0}}$ and $u|_{\cG_{0}^{(2)}}$ are Borel sections and $\alpha$ and $u$ together satisfy the following: 
    \begin{enumerate}
        \item  $\alpha_{g} \circ \alpha_{h} = \Ad(u(g,h)) \circ \alpha_{gh}$ for all $(g,h) \in \mathcal{G}_{0}^{(2)}$. 
        \item $\alpha_{g}(u(h,k))u(g,hk) = u(g,h)u(gh,k)$ for all $(g,h,k) \in \mathcal{G}_{0}^{(3)}$.
        \item $\alpha_{g} = \id$ when $g \in X_{0}$.
        \item $u(g,h) = 1$ when $g \in X_{0}$ or $h \in X_{0}$. 
    \end{enumerate}
    Then $u$ is called a \textit{2-cocycle for $\alpha$} and the pair $(\alpha,u)$ is called a \textit{cocycle action} of $\mathcal{G}$ on the measured field of factors $B$. We call $X_{0}$ a \textit{Borel domain} for the cocycle action. 
\end{definition}

We now define the notion of cocycle conjugacy for two group actions. Here we shall generalize this notion to cocycle actions of groupoids.
\medskip
\begin{definition}
\label{Def: cocycle conjugacy of groupoid actions}
    Let $\cG$ be a discrete measured groupoid with unit space $(X,\mu)$ and $B = (B_{x})_{x \in X}$ and $(D_{x})_{x \in X}$ be measured fields of factors over the standard Borel space $X$. Let $(\alpha,u)$ and $(\beta,v)$ be two cocycle actions of $\cG$ on the fields $B$ and $D$ respectively. Suppose that there is a common Borel domain $X_{0}$ for $(\alpha,u)$, $(\beta,v)$, the measured field of Polish spaces $(\Iso(B_{x},D_{x}))_{x \in X}$ and the measured field of Polish groups $(\cU(D_{t(g)}))_{g \in \cG}$. Then the actions are called \textit{cocycle conjugate} if after possibly shrinking the Borel domain $X_{0}$, there is a \textit{measured field of *-isomorphisms}, i.e., a Borel section $\theta$ of the Borel field $(\Iso(B_{x},D_{x}))_{x \in X_{0}}$ and a Borel section $w$ of $(\cU(D_{t(g)}))_{g \in \cG_{0}}$ satisfying: 
    \begin{enumerate}
        \item $\theta_{t(g)} \circ \alpha_{g} \circ \theta^{-1}_{s(g)} = \Ad(w_{g}) \circ \beta_{g}$ for all $g \in \mathcal{G}_{0}$
        \item $\theta_{t(g)}(u(g,h)) = w_{g}\beta_{g}(w_{h})v(g,h)w^{*}_{gh}$ for all $(g,h) \in \mathcal{G}_{0}^{(2)}$.  
    \end{enumerate}
    If $u = v = 1$ so that $\alpha$ and $\beta$ are genuine actions, then Condition 2 becomes $w_{gh} = w_{g}\beta_{g}(w_{h})$. A measured field of unitaries $w$ satisfying this condition is called a \textit{1-cocycle} for the action $\beta$. 
\end{definition}

Recall that an automorphism $\alpha$ of a von Neumann algebra $B$ is called \textit{free} or \textit{properly outer} if for all $v \in B$, there exists an element $x \in B$ such that $\alpha(x)v \neq vx$.  If $B$ is a factor, then $\alpha$ is free if and only if it is outer. Recall that we denote by $\cG^x_x$ the isotropy groups of the groupoid. 
\medskip 
\begin{definition}
    A cocycle action $(\alpha,u)$ of a discrete measured groupoid $\cG$ with unit space $(X,\mu)$ on a measured field of factors $B = (B_{x})_{x \in X}$ is called \textit{free} if there is a Borel domain $X_{0}$ for the cocycle action such that for all $x \in X_{0}$ and for all $g \in \cG^x_x$,  $\alpha_{g}:B_{x} \rightarrow B_{x}$ is outer. In particular, a cocycle action of $\cG$ is free if and only if the restricted cocycle action of $\Isot(\cG)$ is free.  
\end{definition}

To draw a parallel with the one-one correspondence between discrete measured groupoids and csm inverse semigroups as in Section \ref{Sec: The inverse semigroup - groupoid correspondence}, we now define the notion of cocycle actions of such inverse semigroups. For a csm inverse semigroup $\cI$, and an element $v \in \cI$, we shall denote by $s(v)$ and $t(v)$ the idempotents $v^{-1}v$ and $vv^{-1}$. For any idempotent $e$, we shall denote by $p_{e}$ the corresponding projection in $\Idem(\cI)$.  
\medskip
\begin{definition}
    \label{Def: actions of inverse semigroups}
    Let $\cI$ be a csm inverse semigroups with $\Idem(\cI)$ isomorphic to $\cP(L^{\infty}(X,\mu))$. Then \textit{an action of $\cI$ on a von Neumann algebra $B$} with $L^{\infty}(X,\mu) = \mathcal{Z}(B)$ is a family of maps $\alpha_{v}: Bp_{s(v)} \rightarrow Bp_{t(v)}$ satisfying $\alpha_{v} \circ \alpha_{w} = \alpha_{vw}$ (restricting domains when necessary) for all $v,w \in \cI$ such that $vw \neq 0$ and $\alpha_{e} = 1_{p_{e}}$ for all idempotents $e$. 

    More generally, \textit{a cocycle action of $\cI$ on $B$} consists of such a family of maps $\{\alpha_{v} \; | \; v \in \cI\}$ together with a family of unitaries $(v,w) \mapsto u_{v,w} \in \cU(Bp_{t(vw)})$ such that (restricting domains when necessary) the following conditions are satisfied: 
    \begin{enumerate}
        \item $\alpha_{v} \circ \alpha_{w} = \Ad(u_{v,w})\alpha_{vw}$ for all $v,w \in \mathcal{I}$ 
        \item $\alpha_{v}(u_{w,x})u_{v,wx} = u_{v,w}u_{vw,x}$ for all $v,w,x \in \mathcal{I}$ 
        \item For any idempotent $e \in \mathcal{I}$, $\alpha_{e} = 1p_{e}$  \item For $v,w \in \mathcal{I}$, $u_{v,w} = 1p_{t(vw)}$ if $v$ or $w$ is an idempotent.
    \end{enumerate}
\end{definition}

We shall now define the crossed product of a cocycle action $(\alpha,u)$ of a discrete measured groupoid $\mathcal{G}$ with unit space $(X,\mu)$ on a measured field of factors $B = (B_{x})_{x \in X}$. This will be done in a similar way as crossed products for group actions. We shall prove in what follows that there are partial isometries $\{u(V) \; | \; V \in [[\cG]]\}$ and unitaries $\{u_{V,W} \; | \; V,W \in [[\cG]] \}$ which together with $B$ satisfy the following conditions: 

\begin{enumerate}
    \item $u(V)^{*}u(V) = 1_{s(V)}$ and $u(V)u(V)^{*} = 1_{t(V)}$
    \item $u(V)u(W) = u_{V,W}u(VW)$
    \item $u(V)bu(V)^{*} = \alpha_{V}(b)$ for all $b \in B1_{s(V)}$
    \item $E(bu(V)) = bp_{V \cap X}$ (where $p_{V \cap X}$ is the projection corresponding to the Borel subset $V \cap X$) for $b \in B$ and $V \in [[\cG]]$ is a faithful normal conditional expectation.
\end{enumerate}
\medskip 
\begin{definition}
\label{Def: cocycle crossed product for groupoid actions}
    Let $\cG$ be a discrete measured groupoid with unit space $(X,\mu)$ and $(\alpha,u)$ be a cocycle action of $\cG$ on a measured field $B = (B_{x})_{x \in X}$ of factors. Then the \textit{crossed product von Neumann algebra} denoted by $M = B \rtimes_{(\alpha,u)}\cG$ is the is the unique von Neumann algebra generated by $B$ and partial isometries $u(V)$ for all $V \in [[\cG]]$ satisfying the four conditions above. 
\end{definition}

Now we prove that we can indeed construct such a family of partial isometries. Let $B_{x} \subset \mathcal{B}(\cH_{x})$ and $\cH = (\cH_{x})_{x \in X}$ denote the corresponding measured field of Hilbert spaces. Let $\mathfrak{H} = \int_{X}^{\oplus}\cH_{x}$ denote the direct integral Hilbert space and $B = \int^{\oplus}_{X} B_{x} \subseteq \cB(\mathfrak{H})$ be the direct integral von Neumann algebra. We know that the center of $B$ is $L^{\infty}(X,\mu)$. By Proposition \ref{Prop: full pseudogroup is an inverse semigroup}, we know that $[[\cG]]$ is a csm inverse semigroup. For all $x \in X$, consider now the Hilbert space $\ell^{2}(\cG^{x})$ and let $\cK_{x} = \cH_{x} \otimes \ell^{2}(\cG^{x})$. It follows that $(\cK_{x})_{x \in X}$ is a measured field of Hilbert spaces and let $\mathfrak{K}$ be the direct integral $\mathfrak{K} = \int^{\oplus}_{X} \cK_{x}$. We can see that $B = \int_{X}^{\oplus}B_{x}$ has a faithful normal representation $\pi: B \rightarrow \cB(\mathfrak{K})$ exactly as in the case of group actions defined as follows: 
\begin{align*}
     \pi((b_{x})_{x \in X})((\xi_{x} \otimes \eta_{x})_{x \in X}) = (b_{x}\xi_{x} \otimes \eta_{x})_{x \in X} 
\end{align*}

We note now that $(\alpha,u)$ defines a cocycle action (that we still denote by $(\alpha,u)$) of the inverse semigroup $[[\cG]]$ on $B$. For an idempotent $E \subseteq X$ consider the projection $p_{E}$ as an element of $\cB(\mathfrak{K})$. Note that any element $b \in Bp_{E}$ is a decomposable operator in $\mathcal{B}(p_{E}\mathfrak{K})$ and can be written as $\int^{\oplus}_{E} b_{x}$. Now for $V \in [[\cG]]$, let $p_{s(V)}$ and $p_{t(V)}$ denote the central projections in $B$ corresponding to the Borel subsets $s(V)$ and $t(V)$ of $X$. We define the *-isomorphism $\alpha_{V}: Bp_{s(V)} \rightarrow Bp_{t(V)}$ as follows: 
\begin{align*}
    \alpha_{V}(b) = \int^{\oplus}_{t(V)} \alpha_{g}(b_{s(g)}) \; d\mu(x) 
\end{align*}
Similarly for two elements $V,W \in [[\cG]]$, denote by $g_{h}$ the unique element in $V$ such that $s(g_{h}) = t(h)$ for all elements in the Borel set $E = \{h \in W \; | \; t(h) \in s(V)\}$. It can be checked that the inverse of an injective Borel map is Borel (for example, see \cite[Section 15.A]{Kec95}). This shows since $s|_{E}$ is an injective Borel map, and hence $s^{-1} \circ t |_{E}$ given by $h \mapsto g_{h}$ is Borel. Now we define the 2-cocycle $u_{V,W} \in B1_{t(VW)}$ as follows:   
\begin{align*}
    u_{V,W} = \int^{\oplus}_{t(VW)}u(g_{h},h) \; d\mu(x)
\end{align*}

Now consider for all elements $g \in \cG$ the unitary operator $u_{g}: \cK_{s(g)} \rightarrow \cK_{t(g)}$ given by: 
\begin{align*}
    u_{g}(\xi \otimes \delta_{h}) = \xi \otimes \delta_{gh}
\end{align*}
For any bisection $V \in [[\cG]]$ we can construct a partial isometry $u(V) \in \cB(\mathfrak{K})$ with domain $\int_{s(V)}^{\oplus} \cH_{x} \otimes \ell^{2}(\cG^{x})$ and range $\int_{t(V)}^{\oplus} \cH_{x} \otimes \ell^{2}(\cG^{x})$ given by:
\begin{align*}
    u(V) = \int^{\oplus}_{V} u_{g} \; d\mu(g) 
\end{align*}
The crossed product is then the von Neumann algebra generated by $\pi(B)$ and the partial isometries $u(V) \in \cB(\mathfrak{K})$ for all $V \in [[\cG]]$. One can check that the four conditions above are satisfied, almost by construction. The direct integral von Neumann algebra $\pi(B)$ is a subalgebra of $M$ and from now on we shall abuse notation and use $B$ to denote the subalgebra, the measured field as well as the direct integral. 
\medskip

\begin{remark}
\label{Rmk: actions of groupoids and inverse semigroups are the same}
    Notice that the crossed product construction for a cocycle action of a groupoid $\cG$ with unit space $(X,\mu)$ essentially uses the action of the csm inverse semigroup $[[\cG]]$. Suppose that we started with a cocycle action $(\alpha,u)$ of a csm inverse semigroup $\cI$ on a von Neumann algebra $B$ such that $\cP(\cZ(B))$ isomorphic to $\Idem(\cI)$. Then we can similarly construct the crossed product $B \rtimes_{(\alpha,u)} \cI$. Now, let $\GI$ be the discrete measured groupoid with unit space $(X,\mu)$ from Theorem \ref{Theorem A} such that $\cI$ is isomorphic to $[[\GI]]$. Suppose that the direct integral decomposition of $B$ is $(B_{x})_{x \in X}$. We claim that the cocycle action of $[[\GI]]$ can be canonically lifted to a cocycle action of $\GI$. To see this, we first fix a basis $\cV$ for $\GI$. Let $\cV_{1}$ be the countable set of bisections formed by taking finite products of the bisections in $\cV$. Let us enumerate the elements of $\cV_{1}$ by $\cV_{1} = \{V_{0},V_{1},V_{2},...\}$. By discarding a null set $E_{n}$ from each $V_{n}$ and letting $V_{n}^{0} = V_{n} \backslash E_{n}$, we apply the disintegration of automorphisms (\cite[Corollary X.3.12]{TakesakiVolume2}) and get that the *-isomorphism $\alpha_{V}: B_{s(V)} \rightarrow B_{t(V)}$ disintegrates into a direct integral of *-isomorphisms $\alpha_{g}: B_{s(g)} \rightarrow B_{t(g)}$ for all $g \in V_{n}^{0}$. 
    
    Letting $E = \bigcup_{n} E_{n}$, we now get a *-isomorphism $\alpha_{g}: B_{s(g)} \rightarrow B_{t(g)}$ for all $g \in \GI \backslash E$. Similarly we get a null set $E'$ such that we can define the 2-cocycle $u(g,h)$ for all $g,h \in \GI \backslash E'$. Now taking $g,h,k$ outside the null set $E \bigcup E'$, notice that $\alpha_{g}(u(h,k))u(g,hk) = u(g,h)u(gh,k)$ as the corresponding 2-cocycle relation is satisfied by the action of the inverse semigroup. Similarly we have that $\alpha_{g} \circ \alpha_{h} = \Ad (u(g,h)) \circ \alpha_{gh}$ for all $g,h$ outside $E \bigcup E'$. Similarly one can check that the four points in Definition \ref{2-cocycle on groupoids} are satisfied for a.e. element $g,h \in \GI$. We therefore get a cocycle action of $\GI$ on the measured field $(B_{x})_{x \in X}$ such that $B \rtimes_{(\alpha,u)} \cI$ is isomorphic to $(B_{x})_{x \in X} \rtimes_{(\alpha,u)} \GI$, by construction of the cocycle crossed product. We shall abuse notation throughout this section and use $(\alpha,u)$ to denote the cocycle actions of both the groupoid and the inverse semigroup, depending on the context.
\end{remark}

As before, consider such a cocycle crossed product $M = B \rtimes_{(\alpha,u)} \cG$ represented on the Hilbert space $\mathfrak{K} = \int_{x}^{\oplus} \cK_{x}$ where $\cK_{x} = \ell^{2}(\cG^{x}) \otimes \cH_{x}$ and where $B$ has central decomposition $\int^{\oplus}_{x} B_{x}$ with $B_{x} \subseteq \cB(\cH_{x})$.
Denote by $\mathfrak{H}$ the direct integral $\int^{\oplus}_{X}\cH_{x}$. Since $\cK_{x}$ is also isomorphic to $\ell^{2}(\cG^{x}, \cH_{x})$, for each bisection $V \in [[\cG]]$, consider the partial isometries $P_{V}:\int_{s(V)}^{\oplus} \cK_{x} \rightarrow \int_{t(V)}^{\oplus} \cH_{x}$ where: 
\begin{align*}
    P_{V}(\xi_{s(g)}) = \xi_{t(g)}(g) \text{ for all } g \in V \text{ and } \xi = (\xi_{x})_{x \in s(V)}
\end{align*}
In particular for the unit space $X$, we have a partial isometry $P_{X}: \mathfrak{K} \rightarrow \mathfrak{H}$ given by $P_{X}(\xi_{y}) = \xi_{y}(y)$ for all $y \in X$ where by abusing notation $y$ also denotes the identity element in $\cG^{y}$. Now consider the map $E: M \rightarrow B$ given by $E(y) = P_{X}yP_{X}^{*}$. One can check, exactly as in \cite[Pg 255]{Mercer85} or \cite[Pg 365]{TakesakiVolume1} that $E$ indeed is a faithful normal conditional expectation onto $B$. 

Now we define the following topology on $M$: consider the seminorm $\zeta_{\omega}$ given by $\zeta_{\omega}(x) =\omega \circ E(x^{*}x)^{\frac{1}{2}}$ for a faithful normal state $\omega \in B_{*}$. The locally convex topology generated by the family of seminorms
\begin{align*}
    \{\zeta_{\omega} \; | \; \omega \text{ faithful normal state in} B_{*} \}    
\end{align*}
This is called the $\mathfrak{U}$-topology by Mercer in \cite{Mercer85} and to the author's knowledge, was first documented by Bures in \cite{Bures71}. Recall that every faithful normal state $\phi$ on a von Neumann algebra induces a sharp norm given by  $\|x\|_{\phi}^{\sharp} = \sqrt{|\phi(x^{*}x)| + |\phi(xx^{*})}$. One can check that a net $x_{i} \rightarrow x$ in the $\mathfrak{U}$-topology if and only $\|x_{i} \rightarrow x\|^{\sharp}_{\omega \circ E} \rightarrow 0$ for any faithful normal state $\omega \in B_{*}$. 

Suppose now that we have a symmetric basis $\cV$ for the groupoid $\cG$. For all $x \in M$, let $x_{V} = E(x^{*}u(V))$ for all $V \in \cV$. Exactly the same argument as in \cite[Proposition 3.5]{ber-chak-don-kim} shows that $\sum_{V \in \cV} x_{V}u(V)$ converges to $x$ under the sharp norm for any faithful normal state of the form $E \circ \omega$ on $M$, where $\omega$ is a faithful normal state on $B$. Notice that this is true for a general basis as well, but the calculations are much easier if we consider symmetric bases. This indeed gives a regular inclusion as in the following lemma, and a proof appears in \cite[Section 2.2]{Chakraborty24a}
\medskip
\begin{lemma}
\label{Lemma: B is regular in crossed product M}
    Let $\cG \actson_{(\alpha,u)} B = (B_{x})_{x \in X}$ be a cocycle action as above. Then $B$ is a regular subalgebra of the crossed product $M = B \rtimes_{(\alpha,u)} \cG$.  
\end{lemma}
\medskip
\begin{lemma}
    \label{Lemma: B satisfies RCC in crossed product M}
    Let $\cG \actson_{(\alpha,u)} B = (B_{x})_{x \in X}$ be as above and $M$ be the crossed product. Then the relative commutant condition is satisfied, i.e.,  $B'\cap M = \mathcal{Z}(B) = L^{\infty}(X,\mu)$ if and only if the cocycle action is free. 
\end{lemma}

\begin{proof}
    Let $\cV = \{V_{0} = X, V_{1},V_{2},...\}$ be a basis for the groupoid $\cG$. Then any element $a \in M$ admits a Fourier decomposition and can be written as $\sum_{n}a_{n}u(V_{n})$ where $a_{n} \in B$. Now if $a$ commutes with every element $b \in B$, this implies that: 
    \begin{align*}
        \sum_{n} ba_{n}u(V_{n}) = \sum_{n} a_{n}u(V_{n})b = \sum_{n}a_{n}\alpha_{V_{n}}(b)u(V_{n})
    \end{align*}
    This means $ba_{n} = a_{n}\alpha_{V_{n}}(b)$ for all $b \in B$ and for all $n \in \mathbb{N}$. Now if $a \notin \mathcal{Z}(B)$ then for some $k \neq 0$, we have that $a_{k} \neq 0$. This can happen if and only if every $b \in B$ satisfies $ba_{k} = a_{k}\alpha_{V_{k}}(b)$. Now for ease of notation let us denote $a = a_{k}$, $V = V_{k}$, and without loss of generality, assume that the right central support of $a$ is $t(V)$. Thus by looking at the decompositions of $b$ and $a$, in the ergodic decomposition of $B$, we have that for all $b \in B$, there exists a null set $W_{b} \subset V$ such that for all $g \in V \backslash W_{b}$: 
    \begin{align*}
        b_{t(g)}a_{t(g)} = a_{t(g)}\alpha_{g}(b_{s(g)})   
    \end{align*}
    Now since $B$ has a separable predual, we can pick a countable subset $B_{0}$ in the unit ball of $B$ that is $\sigma$-weakly dense. Let us denote the elements of $B_{0}$ by $\{b_{1},b_{2},...\}$. For each $i$, we then have a null set $W_{n}$ such that outside $W_{n}$, the above equation holds. Then taking $W = \bigcup_{n}W_{n}$, we have that for all $b \in B_{0}$ and for all $g \in V \backslash W$, the equation $b_{t(g)}a_{t(g)} = a_{t(g)}\alpha_{g}(b_{s(g)})$. By the $\sigma$-weak continuity of the maps $B \rightarrow B$ given by $b \mapsto ba$ and $b \mapsto a\alpha_{g}(b)$, the equation holds for all $g \in V \backslash W$ and all $b \in B$. Now suppose that for some $g \in V \backslash W$, $s(g) \neq t(g)$. Let us pick a Borel subset $E \subset X$ such that $s(g) \in E$ and $t(g) \notin E$. Let $p$ be the corresponding central projection in $B$. Therefore it follows that:  
    \begin{align*}
        p_{t(g)}a_{t(g)} = a_{t(g)}\alpha_{g}(p_{s(g)}) \implies 0 = a_{t(g)}
    \end{align*}
    Since we assumed that $a$ is supported on $t(V)$, this is a contradiction and $V$ is a bisection in $\Isot(\cG)$. Now for all $g \in V \backslash W$, since $B_{s(g)}$ is a factor, we can find a unitary $u_{s(g)}$ such that $b_{s(g)}u_{s(g)} = u_{s(g)}\alpha_{g}(b_{s(g)})$ for all $b \in B$, hence contradicting the freeness of $\alpha$.  
    
    Conversely, if the action is not free, there exists a bisection $V$ of $\Isot(\cG)$ and an element $x \in B$ such that $bx = x\alpha_{V}(b)$ for all $b \in B$. Then the element $xu(V)$ does not belong to $L^{\infty}(\cG^{(0)})$ and still commutes with every element in $B$, completing the proof.  
\end{proof}

\begin{lemma}
    \label{Lemma: ergodic groupoid if and only if crossed product factor}
    Let $(\alpha,u)$ be a free cocycle action of $\cG$ on $B = (B_{x})_{x \in X}$ as above and $M$ be the crossed product. Then $\cG$ is ergodic if and only if $M$ is a factor. 
\end{lemma}
\begin{proof}
    This follows immediately from Lemma \ref{Lemma: B satisfies RCC in crossed product M} as the set of elements $a \in L^{\infty}(X,\mu)$ commuting with all of $M$ is precisely $L^{\infty}(X,\mu)^{\cG}$, the $\cG$-invariant functions, and this is equal to $\mathbb{C}$ if and only if $\cG$ is ergodic. 
\end{proof}

From the previous lemmas it is clear that for a discrete measured groupoid $\cG$ with unit space $(X,\mu)$ and a cocycle action $(\alpha,u)$ on a measured field $B = (B_{x})_{x \in X}$ of factors, denoting by $M = B \rtimes_{(\alpha,u)} \cG$ we have that $B$ is a regular subalgebra of $M$ satisfying the relative commutant condition $\mathcal{Z}(B) = B' \cap M$ and with a faithful normal conditional expectation $E: M \rightarrow B$. Now we shall prove a converse of this and construct a groupoid and a cocycle action from such a regular inclusion. 
\medskip 
\begin{theorem}
    \label{Thm: groupoid from regular inclusion}
    Let $B \subset M$ be a regular subalgebra of a von Neumann algebra with a faithful normal conditional expectation $E: M \rightarrow B$. Let $\mathcal{Z}(B) = L^{\infty}(X,\mu)$ and suppose that the relative commutant condition $\mathcal{Z}(B) = B'\cap M$ is satisfied. 
    
    Then there is a discrete measured groupoid $\cG = \cG_{B \subset M}$ with unit space $(X,\mu)$ and a free cocycle action $(\alpha,u)$ of $\cG$ on the central decomposition $B = (B_{x})_{x \in X}$ such that there is a von Neumann algebra isomorphism $\theta: M \rightarrow B \rtimes_{(\alpha,u)} \cG$ with $\theta(B) = B$ and $\theta \circ E = E_{\cG} \circ \theta$ where $E_{\cG}$ is the canonical conditional expectation from the crossed product to $B$. Moreover $\cG_{B \subset M}$ is ergodic if and only if $M$ is a factor. 
\end{theorem} 
\begin{proof}
      Let $B = (B_{x})_{x \in X}$ be the measured field of factors arising from the direct integral decomposition of $B$. Let $\cP$ denote the set of partial isometries $v \in M$ such that $s(v) \coloneqq v^{*}v \in \cZ(B)$, $r(v) \coloneqq vv^{*} \in \cZ(B)$ and $vBv^{*} = Br(v)$. For such a $v \in \cP$, we immediately have \textit{Fact 1}: If $b$ is a unitary in $\cU(Br(v))$, then $vbv^{*} = b'vv^{*}$ for some unitary $b' \in \cU(Br(v))$. We now identify two partial isometries $w \sim v$ in $\cP$ if $s(v)=s(w), r(v) = r(w)$ and there exists a partial isometry $b \in B$ with $b^{*}b = bb^{*} = r(v)$ and $w = bv$ or equivalently if there exists a unitary $b \in \cU(Br(v))$ such that $w = bv$.  \textit{Fact 2}: $v \sim w$ if and only if there exists a unitary $b' \in \cU(Bs(v))$ such that $w = vb'$. To realize this, if indeed $bv = w$ for some $b$, then $v^{*}w = v^{*}bv = b'v^{*}v$ for some $b' \in \cU(Bs(v))$ by Fact 1, and then we have that $w = vv^{*}w = vb'v^{*}v = vb'$. By a symmetric argument for the converse, we have a proof of Fact 2. 
       
       It is easy to check that $\sim$ gives an equivalence relation on $\cP$. Now let us denote the quotient by $\cI = \cP / \sim$. We will show that $\cI$ has a well defined inverse semigroup structure. Indeed if $v_{1} = b_{1}w_{1}$ and $v_{2} = b_{2}w_{2}$ then we can pick a unitary $b \in \cU(Br(w_{1}))$ such that $w_{1}b_{2}w_{1}^{*} = bw_{1}w_{1}^{*}$ and consequently:  
       \begin{align*}
           v_{1}v_{2} = b_{1}w_{1}b_{2}w_{2} = b_{1}w_{1}w_{1}^{*}w_{1}b_{2}w_{2} = b_{1}w_{1}b_{2}w_{1}^{*}w_{1}w_{2} = (b_{1}b)w_{1}w_{2}
       \end{align*}
       Similarly if $v \sim w$ and $w = bv$ then $w^{*} = v^{*}b^{*}$ and by Fact 2, we have that $v^{*} \sim w^{*}$. One can easily check that the idempotents in $\cI$ are precisely the central projections in $B$. Since any two such central projections are equivalent if and only if they are equal, we have that $\Idem(\cI)$ is $\cP(A)$ where $A = L^{\infty}(X,\mu)$ and $\cI$ is a measured inverse semigroup. For a countable family of partial isometries $v_{n}$ with orthogonal source and ranges, the sequence $\sum_{i=1}^{n}v_{k}$ converges in SOT to the limit which is also a partial isometry $v = \sum_{n} v_{n}$ thus making the semigroup complete. This is in fact well defined as if we have $v_{n} = bw_{n}$ for all $n$, then the sources and ranges of $w_{n}$'s are orthogonal and consequently we have $\sum_{n}v_{n} = (\sum_{n}b_{n})(\sum_{n}w_{n})$, as required.  
       
       We call two partial isometries $v,w \in \mathcal{P}$ orthogonal under $E$ if $E(v^{*}w) = 0$. \textit{Fact 3}: Any subset $P \subset \mathcal{P}$ consisting of elements that are pairwise orthogonal under $E$ is countable. To prove this, consider a faithful normal state $\omega_{0}$ on $B$, then $\omega = \omega_{0} \circ E$ is a faithful normal state on $M$. Consider the GNS representation of $M$ on the Hilbert space $\cH = L^{2}(M,\omega)$. Since $M$ has a separable predual, $\cH$ is a separable Hilbert space. For two partial isometries $v,w \in \cP$, if $E(v^{*}w) = 0$, then for the corresponding vectors $\hat{v}$ and $\hat{w}$ in $\cH$, we have that $\innerproduct{\hat{v}}{\hat{w}} = 0$ by definition. Thus the set $\{\hat{v} \; | \; v \in P\}$ consists of pairwise orthogonal vectors in the Hilbert space and by separability it is countable. Thus $P$ is countable, proving Fact 3. By an application of Zorn's Lemma, we can now pick a maximal countable set $P$ in $\cP$ of elements orthogonal under $E$. Let us denote by $\overline{v}$ the equivalence class of a partial isometry $v$ in $\cI$. Taking $v_{0} = 1$, we claim that the corresponding set of elements $\{\overline{v_{0}},\overline{v_{1}},\overline{v_{2}}...\}$ in $\cI$ forms a countable generating set. 
       
       To prove this, we first make \textit{Claim 4}: for all $v \in \cP$, there exists a unique central projection $z \in \cZ(B)$ such that $E(v) = vz = zv$. To prove the claim: notice that for all $b \in Bs(v)$, there exists an element $\alpha(b) \in Br(v)$ such that $vbv^{*} = \alpha(b)$. This implies that $vb = \alpha(b)v$ and hence $E(v)b = \alpha(b)E(v)$. Now one can check that $\alpha: Bs(v) \rightarrow Br(v)$ is a *-isomorphism. Let $z$ be the maximal central projection such that $\alpha|_{Bz}$ is inner. So $\alpha(b) = cbc^{*}$ for all $b \in Bz$ for some unitary $c \in \cU(Bz)$. Then $c^{*}v$ commutes with all elements in $B$, and hence $c^{*}v \in B' \cap M$. By the relative commutant condition, we have that $c^{*}v \in \cZ(B)$. Now we note that for all $b \in Bc^{*}v$, we indeed have the equality $vbv^{*} = cbc^{*}$ and hence by maximality of $z$, we have $c^{*}v = z$. It follows immediately that $E(v) = vz = zv$. Now we make \textit{Claim 5}: for all $v \in \cP$, we have $v = \bigvee_{n}v_{n}E(v_{n}^{*}v)$. By Claim 4, we have that $v_{n}E(v_{n}^{*}v) = r(v_{n})vz_{n} = vz_{n}$ for central projections $z_{n}$. One can check that since $v_{n}$ are orthogonal under $E$, the central projections $z_{n}$ are mutually orthogonal. Thus the only thing left to prove is that $\sum_{n} vz_{n} = v$. To check this suppose that $z^{c} = 1 - \sum_{n}z_{n}$ and suppose $vz^{c} \neq 0$. Then $vz^{c}$ is a partial isometry in $\cP$ which is orthogonal under $E$ to all $v_{n}$'s, thus contradicting the maximality of $P$. Now with the claim proved, let $b_{n} = E(v_{n}^{*}v)$ and note that $b_{n}$'s have pairwise orthogonal domains and ranges. We have that $v_{n}b_{n} \sim v_{n}$ and hence $\overline{v} =\bigvee_{n} \overline{v_{n}}\overline{z_{n}}$ for orthogonal idempotents $z_{n}$. Therefore $\mathcal{I}$ is separable and hence a csm inverse semigroup. 

       Next we construct a cocycle action of $\cI$ on $B$. For a partial isometry $v \in \cP$, we have a *-isomorphism $\alpha_{v}: Bs(v) \rightarrow Br(v)$ defined by $b \mapsto vbv^{*}$. We shall see now that this induces a well defined cocycle action of $\cI$ on $B$. Let $i$ be a lift from $\mathcal{I}$ to $\mathcal{P}$, i.e., for an equivalence class $V \in \mathcal{I}$, we choose an element $v = i(V) \in \mathcal{P}$ such that $v \in \overline{v}$. Moreover let us choose $i$ in such a way that for every $n$, $i(\overline{v_{n}}) = v_{n}$. In what follows, we shall denote the elements of $\cI$ by $\overline{v}$, where $v$ is the unique element such that $i(\overline{v}) = v$. As the lift is typically not unique, $i$ is not necessarily a homomorphism. However, for every pair $\overline{v},\overline{w} \in \mathcal{I}$ there exists a unitary $u_{\overline{v},\overline{w}} \in \cU(Br(v))$ such that $i(\overline{v})i(\overline{w}) = u_{\overline{v},\overline{w}}i(\overline{vw})$. We now define the cocycle action by $\alpha_{\overline{v}}(b) \coloneqq \alpha_{i(\overline{v})}(b) = \alpha_{v}(b)$ for $b \in Br(v)$. One can check that this is indeed well defined. We check now that the 2-cocycle identity is satisfied: 
       \begin{align*}
           \alpha_{\overline{v}}(u_{\overline{w},\overline{x}})u_{\overline{v},\overline{wx}} &= i(\overline{v})u_{\overline{w},\overline{x}}i(\overline{v})^{*}i(\overline{v})i(\overline{wx})i(\overline{vwx})^{*} = i(\overline{v}) i(\overline{w})i(\overline{x})i(\overline{vwx})^{*} \\ &= u_{\overline{v},\overline{w}}i(\overline{vw})i(\overline{x})i(\overline{vwx})^{*} = u_{\overline{v},\overline{x}}u_{\overline{vw},\overline{x}} 
       \end{align*}
       We also note that from the 2-cocycle identity it follows immediately that:  
        \begin{align*}
        \alpha_{\overline{v}} \circ \alpha_{\overline{w}}(b) = \alpha_{i(\overline{v})}\circ \alpha_{i(\overline{w})}(b) = \alpha_{i(\overline{v})i(\overline{w})}(b) = \alpha_{u_{\overline{v},\overline{w}}i(\overline{vw})}(b)  = \Ad (u_{\overline{v},\overline{w}})\alpha_{\overline{vw}}(b). \nonumber
         \end{align*}
        This gives a cocycle action $(\alpha,u)$ of $\cI$ on $B$. By Theorem \ref{Theorem A}, we have a discrete measured groupoid $\GI$ with unit space $(X,\mu)$ such that there is an isomorphism $\gamma: [[\GI]] \rightarrow \cI$ of csm inverse semigroups. As in Remark \ref{Rmk: actions of groupoids and inverse semigroups are the same}, we get a cocycle action $(\alpha,u)$ of $\GI$ on $(B_{x})_{x\in X}$ such that the crossed products $(B_{x})_{x \in X} \rtimes_{(\alpha,u)} \GI$ and $B \rtimes_{(\alpha,u)} \cI$ are isomorphic von Neumann algebras.
        Exploiting the inverse semigroup isomorphism, we know that for every element $U \in [[\GI]]$ there is a *-isomorphism $\alpha_{U}:B1_{s(U)} \rightarrow B1_{t(U)}$ given by $\alpha_{U}(b) = \alpha_{\gamma(U)}(b)$. Corresponding to every pair $U,V \in [[\GI]]$, the unitary element $u(U,V)$ is given by $u_{\gamma(U),\gamma(V)}$. \
        
        We claim that the cocycle action is free. If this were not true, we would have a bisection $V \in [[\Iso(\GI)]]$ with $s(V) = t(V) = E$ such that $V$ does not intersect $\cG^{(0)}$ and a partial isometry $u \in B$ with $s(u) = t(u) = p_{E}$ such that $\alpha_{V}(b) = ubu^{*}$ for all $b \in Bt(V)$. This implies that for the element $v = i(\gamma(V))$, we have that $vbv^{*} = ubu^{*}$ for all $b \in Br(v)$. This would imply that $vu^{*} \in \cZ(Br(v))$, and since $vu^{*} \in \cP$, this in turn implies that $vu^{*}$ is a central projection and hence in particular that $\overline{v}$ is an idempotent in $\cI$. However this is a contradiction as $\gamma^{-1}(\overline{v}) = V$ is not an idempotent. Letting $M_{1} = B \rtimes_{(\alpha,u)} \GI$ we are only left to show that there is an isomorphism $\theta: M_{1} \rightarrow M$ with $\theta(B) = B$. 

        As before let $\omega_{0}$ be a faithful normal state on $B$ and let $\omega = \omega_{0} \circ E$ and $\omega_{1} = \omega_{0} \circ E_{\GI}$ where $E_{\GI}$ is the canonical conditional expectation from the crossed product $M_{1}$ onto $B$.
        Let $\cH = L^{2}(M, \omega)$ and $\cH_{1} = L^{2}(M_{1},\omega_{1})$ be the GNS-representations corresponding to $(M,\omega)$ and $(M_{1},\omega_{1})$. Let $\cK_{1} \subset \cH_{1}$ be the dense subset consisting of the elements of $B$ and the partial isometries $u(V)$ for all $V \in [[\GI]]$. We define a map $\widetilde{U}$ on the dense subset $\cK_{1}$ by $\widetilde{U}(\widehat{u(V)}) \coloneqq \widehat{i(\gamma(V))}$ for all $V \in [[\cG]]$ and $\widetilde{U}(\widehat{b}) \coloneqq \widehat{b}$. Using a countable generating set and a similar argument in the previous parts of the proof, one can easily see that $E_{\GI}(u(V)) = E(i(\gamma(V)))$. As a consequence we immediately have that $\widetilde{U}$ is a norm preserving map between dense subsets of Hilbert spaces. Hence $\widetilde{U}$ extends to a unitary isomorphism between $\cH$ and $\cH_{1}$. We now define the *-isomorphism $\theta: M_{1} \rightarrow M$ by $\theta(x) \coloneqq \widetilde{U}x\widetilde{U}^{*}$. By construction $\theta$ is an isomorphism between $M_{1}$ and $M$ preserving the subalgebra $B$ and satisfying $E \circ \theta = E_{\GI}$. By Lemma \ref{Lemma: ergodic groupoid if and only if crossed product factor}, $M$ is a factor if and only if the groupoid is ergodic.
        \end{proof}

Recall that a discrete measured groupoid is called amenable if there is a norm 1 projection $m: L^\infty(\cG,\mu) \rightarrow L^{\infty}(X,\mu_0)$ such that $m(U \cdot f) = U \cdot m(f)$ for $U \in [[\cG]]$ and the canonical action of $[[\cG]]$ on the space of bounded measurable functions.
\medskip
\begin{proposition}
    \label{Prop: regular inclusion injective iff groupoid amenable}
    The discrete measured groupoid $\cG_{B \subset M}$ in Theorem \ref{Thm: groupoid from regular inclusion} is amenable if and only if $M$ is injective. 
\end{proposition}
\begin{proof}
    If  $\cG \coloneqq \cG_{B \subset M}$ is amenable, then the crossed product $M$ is injective by \cite[Theorem 4.2]{Yamanouchi94}. We prove the converse now: suppose that $\cG$ is a discrete measured groupoid with unit space $(X,\mu)$ and $\alpha$ is an outer action of $\cG$ on an injective von Neumann algebra $B = (B_{x})_{x \in X}$ such that the crossed product $M$ is injective. Letting $(\cH_{x})_{x \in X}$ be a measured field of Hilbert spaces such that $B_{x} \subset \cB(\cH_{x})$, consider the measured field  of Hilbert spaces $(\cK_{x})_{x \in X}$ where $\cK_{x} = \cH_{x} \otimes \ell^{2}(\cG^{x})$. Let $\mathfrak{K}$ be the direct integral of this measured field. By injectivity, we know that there is a conditional expectation $P: \cB(\mathfrak{K}) \rightarrow M$. Consider the von Neumann algebra $A = L^{\infty}(\cG,\mu)$, and notice that $A$ has a faithful normal representation $\pi$ on $\cB(\mathfrak{K})$ given by $\pi(F)(G) = F \cdot G$ for $G \in L^{2}(\cG,\mu) \cong \int^{\oplus}_{X} \ell^{2}(\cG^{x})$. Consider $A$ as a von Neumann subalgebra of $\cB(\mathfrak{K})$. For any element $b \in B$, notice that since $L^{\infty}(\cG, \mu)$ acts only on the second components of the tensor product in $\mathfrak{K}$, we have: 
    \begin{align*}
        bP(f) = P(bf) = P(fb) = P(f)b \text{ for all } f \in A
    \end{align*}
    Hence for all $f \in A$, we have that $P(f) \in B' \cap M = \cZ(B) = L^{\infty}(X,\mu)$. Denoting $m = P|_{A}$, we are only left to check that $m(U \cdot f) = U \cdot m(f)$ for all bisections $U \in [[\cG]]$, but this is clear as $P$ is a conditional expectation.   
\end{proof}

The proof of Theorem \ref{Theorem B} is a combination of Lemma \ref{Lemma: B is regular in crossed product M}, Lemma \ref{Lemma: B satisfies RCC in crossed product M} and Theorem \ref{Thm: groupoid from regular inclusion}. 

\section{Strongly normal subequivalence relations}
\label{Sec: quotients of equivalence relations}

In this section we relate the correspondences established between groupoids and inverse semigroups in Section \ref{Sec: The inverse semigroup - groupoid correspondence} and between their actions and regular inclusions in von Neumann algebras in Section \ref{Sec: Cocycle actions and regular subalgebras} to the setting of equivalence relations. We first define the notion of measured fields of equivalence relations. As in \cite[Definition A.2.1]{Lisethesis}, we call a family of standard Borel spaces $(Y_x)_{x \in X}$ indexed over a standard probability space $(X,\mu)$ a \textit{measured field of standard Borel spaces} if there is a conull subset $X_0$ and a $\sigma$-algebra structure on the disjoint union $Y = \bigsqcup_x Y_x$ such that the projection map $\pi: Y \rightarrow X$ is Borel on $\pi^{-1}(X_0)$ and for all $x \in X_0$, the Borel structures on $\pi^{-1}(\{x\})$ and $Y_x$ coincide. A family of probability measures $(\mu_x \in P(Y_x))_{x \in X}$ will be called a Borel field of probability measures if for any Borel subset $E \subseteq Y$, the map $x \mapsto \mu_x(E \cap Y_x)$ is a Borel map. As usual such a family will be called a measured field if it forms a Borel field on a conull subset of $X$.

\medskip
\begin{definition}
    \label{Def: measured fields of equivalence relations}
    Let $(Z,\eta)$ be a standard probability space and $Y = (Y_{z})_{z \in Z}$ be a Borel field of standard Borel spaces. Suppose that $(\nu_{z})_{z \in Z}$ is a Borel field of probability measures on $(Y_{z})_{z \in Z}$ and $(\cR_{z})_{z \in Z}$ is a family of countable measured equivalence relations on $Y_{z}$, quasi invariant with respect to $\nu_{z}$. Then $(\cR_{z})_{z \in Z}$ is called a \textit{Borel field of countable equivalence relations} if the disjoint union $\cR = \bigsqcup_{z \in Z} \cR_{z}$ is a Borel subset of $Y \times Y$. As usual, such a family will be called a \textit{measured field of countable equivalence relations} if there is a common conull Borel domain $Z_{0} \subset Z$ such that the restriction is a Borel field of equivalence relations on a Borel field of standard Borel spaces quasi-invariant with respect to a Borel field of probability measures. 
\end{definition}

We remark here the definitions considerably simplify if we consider a Borel field of measures on a fixed standard Borel space $Y$. In that case however we need to define things differently for type I and non-type I equivalence relations. Our approach let's us define fields of equivalence relations in a more unified way.  

\medskip
\begin{definition}
    \label{Def: direct integral of equivalence relations}
    Suppose that $(Z,\eta)$ is a standard probability space, $Y = (Y_{z})_{z \in Z}$ is a Borel field of standard Borel spaces, $(\nu_{z})_{z \in Z}$ is a Borel field of probability measures on $Y_{z}$ and $(\cR_{z})_{z \in Z}$ is a Borel field of non-singular countable equivalence relations on the measured field $(Y_{z},\nu_{z})_{z \in Z}$. Consider the standard Borel space $Y$ and note that we can integrate the measures $\nu_{z}$ with respect to $\eta$ to get a $\sigma$-finite Borel measure $\nu$ on $Y$. Let $\cR $ be the equivalence relation on $Y$ consisting of the points $\{(y,y') \in \cR_{z} \text{ for some } z \in Z\}$ and note that $\cR$ is a Borel subset of $Y \times Y$ by Definition \ref{Def: measured fields of equivalence relations}. It is also clear from construction that $\cR$ has countable orbits and is quasi-invariant with respect to $\nu$. Then $\cR$ on $(Y,\nu)$ is called the \textit{direct integral of the equivalence relations} $\cR_{z}$. We shall denote the direct integral decomposition as $\cR = \int^{\oplus}_{Z} \cR_{z} \; d\eta(z)$ as in the case of von Neumann algebras. By abusing notation, we shall often drop the mention to the measure $\eta$ and simply write $\cR = \int^{\oplus}_{Z} \cR_{z}$ when the context is clear.
\end{definition}

The following result gives a counterpart of the ergodic decomposition theorem and was proved by Dang Ngoc Nghiem in \cite{Nghiem75}. Our formulation is similar to the one in \cite[Proposition 3.2]{Feldman-Moore-1} except that we do not assume the field of standard Borel spaces to be constant. We note here that there is a more general ergodic decomposition theorem for measured groupoids with a Haar system in \cite[Theorem 6.1]{Hahn78} which in the case of principal discrete measured groupoids, corresponds to our setting. As noted by Hahn in \cite{Hahn78}, the theorem was also proved independently by Ramsay in \cite{Ramsay80}. Recall that for an equivalence relation $\cR$ on $(X,\mu)$, the bounded measurable functions $L^{\infty}(X,\mu)$ form an abelian von Neumann algebra. As before, we denote by $L^{\infty}(X,\mu)^{\cR}$ the subalgebra of $\cR$-invariant functions.  
\medskip
\begin{theorem}\textup{(c.f. \cite{Nghiem75} , \cite[Proposition 3.2]{Feldman-Moore-1})}
    \label{Thm: equivalence relations admit an ergodic decomposition}
    Let $\cR$ be a countable measured equivalence relation on $(X,\mu)$. Let $(Z, \eta)$ be a standard Borel space such that the abelian von Neumann algebra $L^{\infty}(X,\mu)^{\cR}$ is isomorphic to $L^{\infty}(Z,\eta)$. Then $\cR$ is isomorphic to $\int^{\oplus}_{Z} \cR_{z} \; d\eta(z)$ for a measured field of countable non-singular equivalence relations $(\cR_{z})_{z \in Z}$ on a Borel field of standard Borel spaces and probability measures $(Y_{z},\nu_{z})_{z \in Z}$ such that (denoting the projection $\pi: Y \rightarrow Z$), any invariant Borel subset of $X$ is of the form $\pi^{-1}(A)$ for a Borel set $A \subset Z$ up to measure zero and such that for a.e. $z \in Z$, the equivalence relation $\cR_{z}$ is ergodic. 
\end{theorem} 
\medskip
\begin{remark}
    \label{Rmk: measured fields of Cartan inclusions}
    Let $B = (B_{x})_{x \in X}$ be a measured field of von Neumann algebras and $A = (A_{x})_{x \in X}$ be a family of Cartan subalgebras. Then we shall call the family $(A_{x})_{x \in X}$ \textit{a measured field of Cartan subalgebras} and the family of inclusions $(A_{x} \subset B_{x})_{x \in X}$ a \textit{measured field of Cartan inclusions} if, after restricting to a suitable Borel domain, $A$ is a Borel subset of $B$ with the standard Borel structure in $B$. It can be checked that if $\cR = (\cR_{x})_{x \in X}$ is a measured field of ergodic equivalence relations on $(Y_{x}, \nu_{x})_{x \in X}$, then the corresponding Cartan inclusions $(L^{\infty}(Y_{x},\nu_{x}) \subset L(\cR_{x}))_{x \in X}$ is a measured field of Cartan inclusions. The converse is also true as observed in the proof of \cite[Theorem 1]{Feldman-Moore-2}. For convenience we state this for Cartan inclusions where the corresponding 1-cocycle vanishes, for example in the case of injective factors. Let $(A_{x} \subset B_{x})_{x \in X}$ be a measured field of Cartan inclusions. Then the measured field $(A_{x})_{x \in X}$ is isomorphic to $(L^{\infty}(Y_{x},\nu_{x}))_{x \in X}$ for a measured field of standard Borel spaces and probability measures $(Y_{x},\nu_{x})_{x \in X}$. Assume for a.e. $x \in X$, that $B_{x}$ is hyperfinite. Then we get a measured field of ergodic equivalence relations $(\cR_{x})_{x \in X}$ on $(Y_{x},\nu_{x})_{x \in X}$, a common Borel domain $X_{0}$ and a Borel section $\theta$ of $(\Iso(B_{x}, L(\cR_{x})))_{x \in X_{0}}$ such that $\theta|_{A}$ is a Borel section of  $\Iso((A_{x},L^{\infty}(Y_{x},\nu_{x}))_{x \in X_{0}}$. Of course, this does not rely at all on hyperfiniteness, we just added the condition to avoid messy definitions of `measured fields' of 1-cocycles.     
\end{remark}
\medskip
\begin{definition}
\label{Def: strong normality}
    Let $(X,\mu)$ be a standard Borel space and $\cS \subset \cR$ be equivalence relations on $(X,\mu)$. We denote by $\Aut_{\cR}(\cS)$ the full bisections of $\cR$ that normalize $\cS$, i.e., $\Aut_{\cR}(\cS) = \Aut(\cS) \cap [\cR]$. The sub-equivalence relation $\cS$ is called \textit{strongly normal} if there is a countable sequence of elements $\phi_{n} \in \Aut_{\cR}(\cS)$ such that the union of the graphs of $\phi_{n}$ is equal to $\cR$.  
\end{definition}
\medskip
\begin{remark}
\label{Remark: strong normality is enough on a countable collection}
    In various contexts throughout the literature (for example in \cite{feld-suth-zimm-88} and \cite{feld-suth-zimm-89}), subequivalence relations are defined to be strongly normal if the graphs of the elements in $\Aut_{\cR}(\cS)$ `generate' $\cR$. We clarify this notion here. Recall that in a von Neumann algebra with a separable predual, an arbitrary family of projections $(p_{i})_{i \in I}$ has a unique supremum denoted by $p = \bigvee_{i} p_{i}$. Now let $\mu_{s}$ be the measure on the measured equivalence relation $\cR$ by integrating the counting measure over $(X,\mu)$. Then any Borel subset of $\cR$ gives a projection in the abelian von Neumann algebra $L^{\infty}(\cR,\mu_{s})$. Consider the family of projections $\{p_{\phi} \; | \; \phi \in \Aut_{\cR}(\cS)\}$ and let $p$ be the supremum. Then $p$ is of the form $p_{\psi}$ for some Borel subset $\psi \subseteq \cR$. Strong normality of $\cS \subset \cR$ then means that $\psi = \cR$ up to measure zero. The reason this is equivalent to our definition is essentially because $\cR$ is a standard measure space. For each $n > 0$, we can find a finite subset $F_{n} \subset \Aut_{\cR}(\cS)$ such that denoting the set $\cR_{n} = \cR - \bigcup_{\phi \in F_{n}} \text{graph}(\phi)$, we have that $\mu_{s}(\cR_{n}) < \frac{1}{n}$. Now consider the countable set $F = \bigcup_{n} F_{n}$ and notice that up to measure zero, $\bigcup_{\phi \in F} \text{graph}(\phi) = \cR$. Thus we have that $\cR = \bigcup_{\phi \in \Aut_{\cR}(\cS)} \text{graph}(\phi)$ is equivalent to the existence of a countable subfamily $\phi_{n} \in \Aut_{\cR}(\cS)$ such that $\cR = \bigcup_{n} \text{graph}(\phi_{n})$.      
\end{remark}

The notion of strongly normal sub-equivalence relations corresponds to regular subalgebras in the context of von Neumann algebras. We make this precise now. Suppose $A \subset M$ is a Cartan inclusion isomorphic to $L^{\infty}(X,\mu) \subset L(\mathcal{R})$. Then of course there is a faithful normal conditional expectation from $E_{A}: M \rightarrow A$. Interestingly, for any intermediate von Neumann subalgebra $A \subset B \subset M$, Aoi showed in \cite[Theorem 1.1]{Aoi03} that there exists a unique faithful normal conditional expectation $E_{B}: M \rightarrow B$ and $A$ is a Cartan subalgebra of $B$. From this, it follows that intermediate von Neumann subalgebras $A\subset B \subset M$ are in one-to-one correspondence with subequivalence relations $\mathcal{S} \subset \mathcal{R}$ and the bijection is given by considering $B = L(\mathcal{S})$. We note that $B'\cap M \subset A' \cap M = A \subset B$ and hence the relative commutant condition $B'\cap M = \mathcal{Z}(B)$ is always satisfied in this situation. The following was fitst shown in \cite{Aoi03} for normal inclusions of ergodic equivalence relations. The general case was proven in \cite[Proposition 5.1]{Popa-Shlyakhtenko-Vaes20} (for II$_1$ factors) and \cite[Proposition 2.14]{Chakraborty24a} (in general). 
\medskip
\begin{proposition}
\label{Prop: Strongly normal subequivalence relation iff regular subalgbra}
Let $A \subset M$ be a Cartan inclusion isomorphic to $L^{\infty}(X) \subset L(\cR)$. Let $\cS$ be a subequivalence relation of $\cR$ and let $B = L(\cS)$. Then $B$ is a regular subalgebra of $M$ if and only if $\mathcal{S}$ is strongly normal in $\mathcal{R}$. 
\end{proposition}

In the remainder of this section we shall prove that given such a strongly normal inclusion of equivalence relations, one can construct the quotient groupoid which then admits a cocycle action on the subequivalence relation such that the semi-direct product gives the original equivalence relation. Moreover the groupoid is isomorphic to the groupoid constructed from the corresponding regular subalgebra inclusion. Recall that an isomorphism between two equivalence relations $\cR$ and $\cS$ on $(X,\mu)$ and $(Y,\nu)$ respectively is a nonsingular Borel isomorphism $\alpha: X \rightarrow Y$ such that $(x,y) \in \cR$ if and only if $(\alpha(x),\alpha(y)) \in \cS$. The next proposition can be proven exactly in the same way as \cite[Theorem A.2.25]{Lisethesis} and \cite[Proposition 8.9]{WoutersVaes24} and hence we omit a proof. 
\medskip
\begin{proposition}
    \label{Prop: full groups and isomorphisms are measured fields}
    Let $(Z,\eta)$ be a standard probability space and $(Y,\nu) = (Y_{z},\nu_{z})_{z \in Z}$ and $(X,\mu) = (X_{z},\eta_{z})_{z \in Z}$ be measured fields of standard Borel spaces and probability measures. Let $\nu$ and $\mu$ be the integrals of the fields of probability measures with respect to $\eta$. Let $\cR = (\cR_{z})_{z \in Z}$ and $\cS = (\cS_{z})_{z \in Z}$ be measured fields of ergodic countable measured equivalence relations quasi-invariant with respect to $\nu$ and $\mu$ respectively. Let $(A_{z} \subset B_{z})_{z \in Z}$ and $(C_{z} \subset D_{z})_{z \in Z}$ be the corresponding measured fields of Cartan inclusions. Then: 
    \begin{enumerate}
        \item The family $\Iso(\cR,\cS) = (\Iso(\cR_{z},\cS_{z}))_{z \in Z}$ of isomorphisms between the equivalence relations is a measured field of Polish spaces. Similarly $\Iso(A \subset B, C \subset D) = (\Iso(A_{z}\subset B_{z}, C_{z} \subset D_{z}))_{z \in Z}$ is a measured field of Polish spaces. Moreover the natural map $\Iso(\cR,\cS) \rightarrow \Iso(A \subset B, C \subset D)$ is a Borel map, after possibly shrinking the common Borel domain.  
        \item $\Aut(\cR) = (\Aut(\cR)_{z})_{z \in Z}$ and $[\cR] = ([\cR_{z}])_{z \in Z}$ are both measured fields of Polish groups. Similarly $\Aut(A \subset B) = (\Aut(A_{z} \subset B_{z}))_{z \in Z}$ and $\cN_{B}(A) = (\cN_{B_{z}}(A_{z}))_{z \in Z}$ is a measured field of Polish groups. Moreover the natural maps $\Aut(\cR) \rightarrow \Aut(A \subset B)$ and $[R] \rightarrow \cN_{B}(A)$ are Borel after possibly shrinking the common Borel domains.
    \end{enumerate}
\end{proposition}
\medskip
\begin{definition}
    \label{Def: Groupoid actions on fields of equivalence relations}
    Let $(Z,\eta)$ be a standard Borel space and $\cS = (\cS_{z})_{z \in Z}$ be a measured field of ergodic equivalence relations on a measured field of standard Borel spaces and quasi-invariant probability measures $Y = (Y_{z},\nu_{z})_{z \in Z}$. Let $\cG$ be a discrete measured groupoid with unit space $(Z, \eta)$ and $\cI$ be a csm inverse semigroup with $\Idem(\cI)$ isomorphic to $\cP(A)$ where $A = L^{\infty}(Z,\eta)$. Suppose $Z_{0} \subset Z$ is a common Borel domain for the measured fields $(\Iso(\cS_{s(g)},\cS_{t(g)}))_{g \in \cG}$ and $([\cS_{t(g)}])_{(g,h) \in \cG^{(2)}}$ and let $\cG_{0} = \cG|_{Z_{0}}$. Then a \textit{cocycle action} $(\alpha,\Phi)$ of $\cG$ on $\cS$ is given by a Borel section $\alpha$ of $(\Iso(\cS_{s(g)}, \cS_{t(g)}))_{g \in \cG_{0}}$ and a Borel section $\Phi$ of $([\cS_{t(g)}])_{(g,h) \in \cG_{0}^{(2)}}$ such that the following are satisfied: 
    \begin{enumerate}
        \item  $\alpha_{g} \circ \alpha_{h} = \Phi_{g,h} \circ \alpha_{gh}$ for all $(g,h) \in \mathcal{G}_{0}^{(2)}$. 
        \item $\alpha_{g} \circ \Phi_{h,k} \circ \alpha_{g^{-1}} = \Phi_{g,h} \circ \Phi_{gh,k} \circ \Phi_{g,hk}^{-1}$ for all $(g,h,k) \in \mathcal{G}_{0}^{(3)}$.
        \item $\alpha_{g} = 1$ when $g \in Z_{0}$.
        \item $\Phi_{g,h} = 1$ when $g \in Z_{0}$ or $h \in Z_{0}$. 
    \end{enumerate}
    Similarly a \textit{cocycle action} $(\alpha, \Phi)$ of $\cI$ on $\cS$ is given by a family of isomorphisms $\alpha_{v}: \cS|_{s(v)} \rightarrow \cS|_{r(v)}$ for all $v \in \cI$ and a family of partial isomorphisms $\Phi_{v,w} \in [\cS|_{r(v)}]$ such that by restricting domains whenever necessary, the following conditions are satisfied : 
    \begin{enumerate}
        \item $\alpha_{v} \circ \alpha_{w} = \Phi_{v,w} \circ \alpha_{vw}$ for all $v,w \in \cI$.
        \item $\alpha_{v} \circ \Phi_{w,x} \circ \alpha_{v}^{-1} = \Phi_{v,w} \circ \Phi_{vw,x} \circ \Phi_{v,wx}^{-1}$.
        \item For any idempotent $e$, we have $\alpha_{e} = 1_{s(e)}$.
        \item If $v$ or $w$ is an idempotent, then $\Phi_{v,w} = 1_{r(vw)}$.
    \end{enumerate}
\end{definition}

Now for such a cocycle action, we define a semi-direct product equivalence relation which is the counterpart of the crossed product in the setting of von Neumann algebras. 
\medskip
\begin{definition}
    \label{Def: semi-direct product equivalence relations}
    Let us assume the notation of Definition \ref{Def: Groupoid actions on fields of equivalence relations}. Then for a groupoid action, we define the \textit{semi-direct product equivalence relation} $\cR_{\cG} = \cS \rtimes_{(\alpha,\Phi)} \cG$ on the standard Borel space $(Y,\nu)$ by putting $y_{2} \sim_{\cR_{\cG}} y_{1}$ if and only if there exists $g \in \cG$ with $s(g) = \pi(y_{1})$ and $t(g) = \pi(y_2)$ and $(g \cdot y_{1}, y_{2}) \in \cS_{z_{2}}$. Similarly for an inverse semigroup action we define the \textit{semi-direct product equivalence relation} $\cR_{\cI} = \cS \rtimes_{(\alpha,\Phi)} \cI$ on $(Y,\nu)$ as the equivalence relation generated by the graphs of the elements $\{\alpha_{v} \; | \; v \in \cI \}$. As in Remark \ref{Remark: strong normality is enough on a countable collection}, this is equivalent to $\cR$ being the union of the graphs of $\alpha_{v_{n}}$ for a countable generating set $\{v_{n}\}$ of $\cI$. It can be checked that $\cR_{\cG}$ and $\cR_{\cI}$ are both nonsingular countable measured equivalence relations and that $\cR_{\cG} = \cR_{[[\cG]]}$.   
\end{definition}
\medskip
\begin{definition}
    \label{Def: outerness of groupoid action on equivalence relations }
    A cocycle action $(\alpha, \Phi)$ of a discrete measured groupoid $\cG$ with unit space $(Z, \eta)$ on a measured field of ergodic equivalence relations $(\cS_{z}, \nu_{z})_{z \in Z}$ on a measured field of standard Borel spaces $(Y_{z})_{z \in Z}$ is called \textit{free} if for a.e. $z \in Z$ an all $g \in \cG$, we have that $(y,\alpha_{g}(y)) \notin \cS_{s(g)}$ for a.e. $y \in Y$. 
\end{definition}

Notice that letting $(A_{z} \subset B_{z})_{z \in Z}$ be the corresponding measured field of Cartan inclusions, every such cocycle action $(\alpha,\Phi)$ of a discrete measured groupoid on a measured field of equivalence relations $(\cS_{z})_{z \in Z}$ induces a corresponding cocycle action of $\cG$ on the corresponding measured field of Cartan inclusions $(A_{z} \subset B_{z})_{z \in Z}$ essentially by Proposition \ref{Prop: full groups and isomorphisms are measured fields}. Indeed for a common Borel domain $Z_{0}$, the Borel section $\alpha$ induces a Borel section (which we still call $\alpha$ by abusing notation) of the Borel field $(\Iso(A_{s(g)} \subset B_{s(g)}, A_{t(g)} \subset B_{t(g)}))_{g \in \cG_{0}}$ and similarly the section $\Phi$ induces a Borel section $u$ of the Borel field $(\cN_{B_{t(g)}}(A_{t(g)}))_{(g,h) \in \cG_{0}^{(2)}}$ such that all the points in Definition \ref{2-cocycle on groupoids} are satisfied by $(\alpha,u)$. Hence a cocycle action $(\alpha,\Phi)$ as in Definition \ref{Def: Groupoid actions on fields of equivalence relations} induces a cocycle action $(\alpha,u)$ of $\cG$ on a field of Cartan inclusions $(A_{z} \subset B_{z})_{z \in Z}$. Keeping this in mind, we abstractly define a \textit{cocycle action} $(\alpha,u)$ of a discrete measured groupoid $\cG$ with unit space $(Z,\eta)$ on a measured field of Cartan inclusions $(A_{z} \subset B_{z})_{z \in Z}$ as:   
\begin{enumerate}
    \item a measured field $(\alpha_{g})_{g \in \cG}$ of *-isomorphisms $\alpha_{g}:B_{s(g)} \rightarrow B_{t(g)}$ such that $\alpha_{g}(A_{s(g)}) = A_{t(g)}$. We will call this a measured field of *-isomorphisms on Cartan inclusions $(A_{z} \subset B_{z})_{z \in Z}$
    \item a measured field of unitaries $\cG^{(2)} \ni (g,h) \mapsto u_{g,h} \in \cN_{B_{t(g)}}(A_{t(g)})$,
\end{enumerate}
such that the pair $(\alpha,u)$ satisfies the conditions of Definition \ref{2-cocycle on groupoids}. It is also immediate from definition that $(\alpha,\Phi)$ is free if and only if the corresponding cocycle action $(\alpha,u)$ is free.  
\medskip
\begin{definition}
    \label{Def: cocycle conjugacy for groupoid actions on equivalence relations}
    Let $\cG$ be a discrete measured groupoid with unit space $(Z,\eta)$ and $\cR = (\cR_{z})_{z \in Z}$ and $\cS = (\cS_{z})_{z \in Z}$ be two measured fields of ergodic equivalence relations. Suppose that $(A_{z} \subseteq B_{z})_{z \in Z}$ and $(C_{z}\subset D_{z})_{z \in Z}$ are the corresponding measured fields of Cartan inclusions. Two cocycle actions $(\alpha,u)$ and $(\beta,v)$ of $\mathcal{G}$ on the measured fields $(A_{z} \subset B_{z})_{z \in Z}$ and $(C_{z} \subset D_{z})_{z \in Z}$ respectively are said to be \textit{cocycle conjugate} if there exists a common Borel domain $Z_{0}$, a Borel section $\theta$ of $\Iso(A_{z} \subset B_{z}, C_{z} \subset D_{z})_{z \in Z_{0}}$ and a Borel section $w$ of $(\cN_{B_{t(g)}}(A_{t(g)}))_{g \in \cG_{0}}$ that satisfy:
 \begin{itemize}
     \item  $\theta_{t(g)} \circ \alpha_{g} \circ \theta^{-1}_{s(g)} = \Ad(w_{g}) \circ \beta_{g}$ for all $g \in \mathcal{G}_{0}$.
     \item  $\theta_{t(g)}(u(g,h)) = w_{g}\beta_{g}(w_{h})v(g,h)w^{*}_{gh}$ for all $(g,h) \in \mathcal{G}_{0}^{(2)}$.
 \end{itemize}
\end{definition}
\medskip
\begin{proposition}
    \label{Prop: semi-direct product of equivalence relation is crossed product of vNa's}
    Let $(Z,\eta)$ be a standard Borel space and $Y = (Y_{z})_{z \in Z}$ be a measured field of standard Borel spaces. Let $(\cS_{z},\nu_{z})_{z \in Z}$ be a measured field of ergodic equivalence relations on the field $(Y_{z})_{z \in Z}$ and let $\nu$ be the integral of $(\nu_{z})_{z \in Z}$ with respect to $\eta$. Let $\cG$ be a discrete measured groupoid with unit space $(Z, \eta)$ and $(\alpha,\Phi)$ be a free cocycle action of $\cG$ on the measured field $(\cS_{z},\nu_{z})_{z \in Z}$. Suppose $\cS$ on $(Y, \nu)$ is the direct integral of the field and $\cR = (\cS_{z})_{z \in Z} \rtimes_{(\alpha,\Phi)} \cG$. 
    
    Let $B_{z} = L(\cS_{z})$ and let $B$ denote the direct integral of the measured field of factors $(B_{z})_{z \in Z}$. Let $A_{z} = L^{\infty}(Y_{z},\nu_{z})$ be the Cartan subalgebra and let $A \subset B$ denote the direct integral of the field $(A_{z})_{z \in Z}$. Let $(\alpha,u)$ be the free cocycle action of $\cG$ on $B$ induced by $(\alpha,\Phi)$. Then letting $M \coloneqq B \rtimes_{(\alpha,u)} \cG$, $A$ is a Cartan subalgebra of $M$ and there is a *-isomorphism $\theta: L(\cR) \rightarrow M$ such that $\theta(b) = b$ for all $b \in B$. 
\end{proposition}
\begin{proof}
    Let $N = L(\cR)$, we will now construct a unitary operator realizing an isomorphism between the Hilbert spaces $\mathfrak{H} = \int^{\oplus}_{Z} L^{2}(\cS_{z}) \otimes \ell^{2}(\cG^{z})$ and $\mathfrak{K} = L^{2}(\cR)$. Let $U: \mathfrak{K} \rightarrow \mathfrak{H}$ be the map defined on $\xi \in \mathfrak{K}$ as follows: 
    \begin{align*}
        U(\xi)(((y_{z},x_{z}) \otimes g_{z})_{z \in Z}) = \int_{Z} \xi(g_{z}^{-1}y_{z}, x_{z}) \; d\eta(z) 
    \end{align*}
    where $g_{z} \in \cG^{z}$ and $(x_{z},y_{z}) \in \cS_{z}$ for all $z \in Z$. It is a routine check that $U(\xi)$ extends to a unitary operator giving an isomorphism between $\mathfrak{K}$ and $\mathfrak{H}$. Notice that for $\xi \in L^{2}(\cS) \subset \mathfrak{K}$, we have that $U(\xi) = \int^{\oplus}_{Z}\xi_{z}$. We claim that the map $x \mapsto UxU^{*}$ defines a *-isomorphism $\theta: N \rightarrow M$. Note that automatically by the previous observation, $\theta(b) = (b_{z})_{z \in Z}$ for all $b \in L(\cS)$ where $(b_{z})_{z \in Z}$ is the unique decomposition of $b$ in $\int^{\oplus}_{Z} L(\cS_{z})$. 
    
    Since $\cR$ is generated by the graphs of $\alpha_{V}$ for all $V \in [[\cG]]$, consider the elements in $[[\cR]]$ given by $\{(x,\alpha_{V}(x)) \; | \; x \in s(V)\}$ and let $\phi_{V}$ be their characteristic functions. We know that $\phi_{V} \in L^{2}(\mathfrak{K})$ acts as follows: 
    \begin{align*}
        \phi_{V}(\eta)(y,x) = (\phi_{V} * \eta) (y,x) = \sum_{w \in \cR(y)} \phi_{V}(y,w) \eta(w,x) = \eta(h^{-1}y,x)
    \end{align*}
    where $h \in V$ is the unique element with $t(h) = \pi(y)$. Let $u(V)$ be the corresponding partial isometries generating $M$. We claim that for all $V \in [[\cG]]$, $\theta(\phi_{V}) = u(V)$. Let us denote by $h_{z}$ the unique element in $V$ with $s(h_{z}) = z$. We now calculate this on a generating set as follows: 
    \begin{align*}
        u(V)U(\xi)((y_{z},x_{z}) \otimes g_{z}) &= U(\xi)((y_{z},x_{z}) \otimes h_{z}g_{z}) \\ = \int_{Z} \xi(h_{z}^{-1}g_{z}^{-1}y_{z},x_{z}) \; d\eta(z)
        &= \int_{Z}(\phi_{V} * \xi)(g_{z}^{-1}y_{z},x_{z}) \; d\eta(z) \\ = U\phi_{V}(\xi)&((y_{z},x_{z}) \otimes g_{z})
    \end{align*}
    Hence $\theta(\phi_{V}) = u(V)$ for all $V \in [[\cG]]$ and thus the unitary $U$ implements an isomorphism between dense *-subalgebras of $N$ and $M$, which extends to an isomorphism between the corresponding von Neumann algebras, as claimed.
\end{proof}

Combined with Lemma \ref{Lemma: ergodic groupoid if and only if crossed product factor} and Proposition \ref{Prop: Strongly normal subequivalence relation iff regular subalgbra}, this immediately gives the following corollary.
\medskip
\begin{corollary}
    \label{corollary:strong normality in semi-direct product}
     With the notation in Proposition \ref{Prop: semi-direct product of equivalence relation is crossed product of vNa's}, we have that $\cS \subset \cR$ is a strongly normal subequivalence relation. Moreover if $(\alpha,\Phi)$ is free, then $\cR$ is ergodic.  
\end{corollary}

We have shown now that given a cocycle action of a discrete measured groupoid, we get a strongly normal inclusion of equivalence relations and a corresponding regular inclusion of von Neumann algebras. In Section \ref{Sec: Cocycle actions and regular subalgebras}, we proved a converse to this: a construction of a groupoid and a cocycle action from a regular inclusion. Now we show that the associated groupoid can be constructed directly from a strongly normal inclusion of equivalence relations as a `quotient groupoid' together with a cocycle action such that the semi-direct product is isomorphic to the original equivalence relation, as one would expect. Now we describe the construction of such a quotient groupoid. 
\medskip
\begin{construction}
    \label{Construction: quotient inverse semigroup and groupoid}   
    Let $\cR$ be an ergodic equivalence relation on $(X,\mu)$ and $\cS$ be a strongly normal subequivalence relation. Let $(Z, \eta)$ be a standard probability space such that $L^{\infty}(Z, \eta)$ is isomorphic to the abelian von Neumann algebra $L^{\infty}(X,\mu)^{\cS}$ and let $\cS = \int_{Z}^{\oplus} \cS_{z} \; d \eta(z)$ be the ergodic decomposition of $\cS$ such that $\cS_{z}$ is an ergodic equivalence relation on measured field of standard Borel spaces $(Y_{z},\nu_{z})$. We then have that $(X,\mu)$ is isomorphic to $(Y, \nu)$ with $\nu = \int_{Z}^{\oplus}\nu_{z} \; d\eta(z)$. By a \textit{partial automorphism} of $\cS$, we mean a nonsingular Borel bijection $\psi: \pi^{-1}(A)  \rightarrow \pi^{-1}(B)$ for non-null Borel subsets $A,B \subset Z$ such that $(x,x') \in \cS$ if and only if $(\psi(x),\psi(x')) \in \cS$. For such a partial automorphism $\psi$, we denote by $D(\psi)$ the domain $\pi^{-1}(A)$ and by $R(\psi)$ the range $\pi^{-1}(B)$. Let $\cP$ be the set of all partial automorphisms $\psi$ such that for a.e. $x \in D(\psi)$, we have that $(x, \psi(x)) \in \cR$. As in the case of regular inclusions, we say that two such partial isomorphisms are equivalent and denote $\phi \sim \psi$ if up to a null set they have the same source and range and there is a partial isomorphism $\omega \in [[\cS]]$ with $D(\xi) = R(\xi)$ such that $\phi = \omega \circ \psi$. Just as in the case of von Neumann algebras, one can check that this is indeed an equivalence relation. 

    We will now see that the quotient $\cI \coloneqq \cP/ \sim$ is a csm inverse semigroup. Exactly as in the von Neumann algebra setting, one can check that the equivalence relation $\sim$ is preserved under taking products and inverses in $\cP$ and consequently $\cI$ has an inverse semigroup structure. We shall denote the equivalence class of an element $\psi \in \cP$ in $\cI$ by $\overline{\psi}$. The set of idempotents in $\cP$ are precisely the `central projections', i.e., the maps $1_{\pi^{-1}(B)}$ for non-null Borel subsets $B \subset Z$. Note that two such idempotents are equivalent under $\sim$ if and only if they are projections to the same Borel subsets of $Z$ up to measure zero. Therefore $\cI$ is measured and the $\Idem(\cI)$ is isomorphic to $\cP(L^{\infty}(Z,\eta))$. For a sequence of elements $\psi_{n} \in \cP$ with orthogonal domains and ranges, the join $\psi = \bigvee_{n} \psi_{n}$ defined pointwise is still a partial Borel isomorphism in $\cP$ and hence $\mathcal{I}$ is complete. Now let $\phi_{n} \in \Aut_{\cR}(\cS)$ be a countable set such that the graphs generate $\cR$. Given any $\psi \in \cP$, we have that $\text{graph}(\psi) = \bigcup_{n} (\text{graph}(\psi) \cap \text{graph}(\phi_{n}))$ up to measure zero. Since $\phi$ and $\psi$ are both partial isomorphisms, by taking the projection to the first coordinate we have a Borel subsets $A_{n} \subset Z$ such that $\text{graph}(\psi) \cap \text{graph}(\phi_{n}) = \text{graph}(\phi_{n}|_{\pi^{-1}(A_{n})}) = \text{graph}(\psi_{n}|_{\pi^{-1}(A_{n})})$. Thus $\psi = \bigvee_{n} (\phi_{n} \circ 1_{\pi^{-1}(A_{n})})$ and since multiplication is preserved by $\sim$, we have that $\cI$ is separable and hence a csm inverse semigroup. 
\end{construction}
\medskip
\begin{definition}
    \label{Def: quotient semigroup and groupoid for strongly normal inclusion}
    Le $\cR$ be an ergodic equivalence relation on $(X,\mu)$ and $\cS \subset \cR$ be a strongly normal subequivalence relation. Let $(Z, \eta)$ be a standard Borel space such that $L^{\infty}(Z, \eta)$ is isomorphic to $L^{\infty}(X,\mu)^{\cS}$. The csm inverse semigroup $\cI_{\cS \subset \cR}$ with $\Idem(\cI_{\cS \subset \cR})$ isomorphic to $\cP(L^{\infty}(Z,\eta))$ constructed above will be called the \textit{quotient inverse semigroup} for the strongly normal inclusion. The associated discrete measured groupoid $\cG_{\cS \subset \cR}$ with unit space $(Z,\eta)$ by Theorem \ref{Theorem A} will be called the \textit{quotient groupoid} of the inclusion $\cS \subset \cR$.   
\end{definition}
\medskip
\begin{proposition}
\label{Prop: quotient semigroup same as regular inclusion semigroup}
    Let $\cS \subset \cR$ be a strongly normal inclusion (with $\cR$ ergodic) and let $\cI = \cI_{\cS \subset \cR}$ be the quotient inverse semigroup. Let $B = L(\cS) \subset L(\cR) = M$ be the corresponding regular inclusion and let $\cJ = \cI_{B \subset M}$ be the csm inverse semigroup associated to the inclusion. Then $\cI$ is isomorphic to $\cJ$ as csm inverse semigroups and consequently $\cG_{\cS \subset \cR}$ is isomorphic to $\cG_{B \subset M}$.  
\end{proposition}
\begin{proof}
    Let us assume the same notation as in Construction \ref{Construction: quotient inverse semigroup and groupoid}. Suppose that $\mathcal{Q}$ is the set of partial isometries $v$ in $M$ such that $s(v)=v^{*}v$ and $r(v)= vv^{*}$ belong to $\mathcal{Z}(B)$ and $vBv^{*}=Br(v)$. Recall that $\cJ$ is the inverse semigroup arising from the inclusion $B \subset M$ after identifying $v \sim w$ in $\cQ$ if and only if $v = bw$ for some partial isometry $b \in B$ with $s(b)=t(b) = t(w)$. Let $\phi \in \mathcal{P}$, then the graph of $\phi$ denoted by $\text{graph}(\phi)$ is a Borel subset of $\mathcal{R}$. We denote the characteristic function of $\text{graph}(\phi)$ in $M_{b}(\mathcal{R})$ by $S_{\phi}$ and denote the corresponding bounded operator $L_{S_{\phi}}$ on $L^{2}(\mathcal{R},\mu)$ as $u_{\phi}$. We define a map $\Omega: \mathcal{P} \rightarrow \cQ$ by $\Omega(\phi) = u_{\phi}$. Since $\phi \in [[\mathcal{R}]]$, we have that $u_{\phi}$ is a partial isometry such that $u_{\phi}^{*}u_{\phi}$ and $u_{\phi}u_{\phi}^{*}$ are projections in $L^{\infty}(X,\mu)$ corresponding to the characteristic functions of the domain $D(\phi)$ and range $R(\phi)$. Since $\phi \in \Aut(\cS)$, we have that $u_{\phi}$ normalizes $L(\mathcal{S})$, and $\Omega$ is well defined. Now suppose $\phi = \omega \circ \psi$ for some $\omega \in [[\mathcal{S}]]$, then we have $u_{\phi} = u_{\omega \circ \psi} = u_{\omega} \circ u_{\psi}$. Since $u_{\omega}$ is a partial isometry in $L(\cS)$, we have that $u_{\phi}\sim u_{\psi}$. So we have that $\Omega$ passes to the quotients $\Omega: \cI \rightarrow \cJ$. We can check that $\Omega$ preserves multiplication and inverses as: 
    \begin{align*}
        \Omega(\overline{\phi \circ \psi}) = \overline{u_{\phi_{\circ \psi}}} &= \overline{u_{\phi}u_{\psi}} = \overline{u_{\phi}}\overline{u_{\psi}} = \Omega(\overline{\phi}) \circ \Omega(\overline{\psi}) \\
        \Omega(\overline{\phi^{-1}}) &= \overline{u_{\phi^{-1}}} = \overline{u_{\phi}^{*}} = \Omega(\overline{\phi})^{-1}   
    \end{align*}
    By definition $\Omega$ restricted to the idempotents induces the identity map on the lattices of idempotents. What remains to show that $\Omega$ is a lattice isomorphism is that it is indeed a bijection. Let $v$ be a representative in $\cQ$ for the element $\overline{v} \in \cJ$. Exactly as in the proof of Proposition \ref{Prop: Strongly normal subequivalence relation iff regular subalgbra}, we can get a unitary $b \in \cU(Br(v))$ such that $bv \in M$ is a partial isometry satisfying $bvAv^{*}b^{*} = Ar(v)$ for the Cartan subalgebra $A = L^{\infty}(X,\mu)$. Therefore $bv$ is of the form $u_{\omega}u_{\phi}$ for some partial isomorphism $\phi \in \cP$ and $\omega \in [[\cS]]$. Since $b \sim v$, we have that $\Omega(\overline{\phi}) = \overline{v}$, thus proving surjectivity. To check that $\Omega$ is injective, suppose $\phi$ and $\psi$ are elements of $\mathcal{P}$ such that $\Omega(\overline{\phi}) = \Omega(\overline{\psi})$, hence $u_{\phi}= bu_{\psi}$ for some partial isometry $b \in B$. We note then that $b = u_{\phi}u_{\psi}^{*} = u_{\phi \circ \psi^{-1}}$. Since $u_{\phi \circ \psi^{-1}} \in B$ and $B = L(\mathcal{S})$, we have that $\phi \circ \psi^{-1} \in [[\mathcal{S}]]$ and hence $\overline{\phi} = \overline{\psi}$.
\end{proof}

To close the circle of ideas, we state the following theorem which gives a complete correspondence between strongly normal inclusions and cocycle actions of groupoids, such that it canonically fits in the framework of regular inclusions with a Cartan subalgebra. This should be thought of as an equivalence relation version of Theorem \ref{Thm: groupoid from regular inclusion}. 
\medskip
\begin{theorem}
    \label{Thm: strong normailty correspondence with cocycle actions}
    Let $\cR$ be an ergodic equivalence relation on $(X,\mu)$. Let $\cS \subset \cR$ be a strongly normal subequivalence relation with $L^{\infty}(Z,\eta)$ isomorphic to $L^{\infty}(X,\mu)^{\cS}$. Then there is a free cocycle action $(\alpha, \Phi)$ of the quotient groupoid $\cG = \cG_{\cS \subset \cR}$ on the measured field of equivalence relations $(\cS_{z})_{z \in Z}$ obtained from the ergodic decomposition of $\cS$, such that the semi-direct product $(\cS_{z})_{z \in Z} \rtimes_{(\alpha,\Phi)} \cG$ is isomorphic to $\cR$. 
\end{theorem}

\begin{proof}
    Let $\cP, \cQ, \cI, \cJ$ be as in Construction \ref{Construction: quotient inverse semigroup and groupoid} and let $\Omega: \cI \rightarrow \cJ$ be the lattice isomorphism in Proposition \ref{Prop: quotient semigroup same as regular inclusion semigroup}. For any $\psi \in \cP$, consider the map $\alpha_{\psi}: D(\psi) = R(\psi)$ given by $\alpha_{\psi}(x) = \psi(x)$. Exactly as in the proof of Theorem \ref{Thm: groupoid from regular inclusion}, we can take a lift $\cI \rightarrow \cP$ and the non-uniqueness of the lift gives a 2-cocycle action $(\alpha, \Phi)$ of $\cI$ on $\cS$. Now note that $\alpha_{\phi}$ induces the *-isomorphism $\alpha_{v}: Bs(v) \rightarrow Br(v)$ given by $b \mapsto vbv^{*}$ where $v = u_{\psi} \in \cQ$. Similarly for $\psi = u_{v}$ and $\xi = u_{w}$, we have that $\Phi_{\psi,\xi} \in [\cS_{R(\psi)}]$ induces an inner automorphism $\Ad(u_{v,w}): Br(v) \rightarrow Br(v)$ for a corresponding unitary $u_{v,w} \in Br(v)$. It follows from $(\alpha,u)$ being a cocycle action of $\cI$ that via the lattice isomorphism $\Omega$, the pair $(\alpha,u)$ is a cocycle action of $\cJ$ on $B$. Also, note that taking $A = L^{\infty}(X,\mu)$, $\alpha_{v}(A) = A$ for all $v \in \cJ$ and $\Omega_{v,w}$ normalizes $A$ for all $v,w \in \cJ$. Note that Theorem \ref{Thm: groupoid from regular inclusion} does not depend on the choice of lifts and using the theorem, we get an *-isomorphism $\theta: L(\cS) \rtimes_{(\alpha,u)} \cG \rightarrow L(\cR)$ such that $\theta(L(\cS)) = L(\cS)$ and $\theta(A) = A$, where $\cG$ is the unique groupoid from Proposition \ref{Prop: quotient semigroup same as regular inclusion semigroup}. By Proposition \ref{Prop: semi-direct product of equivalence relation is crossed product of vNa's}, we have another *-isomorphism $\theta': L(\cS \rtimes_{(\alpha,\Phi)} \cG) \rightarrow L(\cS) \rtimes_{(\alpha,u)} \cG$ that preserves $L(\cS)$ and $A$. Thus $\theta \circ \theta' : L(\cS \rtimes_{(\alpha,\Phi)} \cG) \rightarrow L(\cR)$ is a *-isomorphism fixing $A$. By the uniqueness part of \cite[Theorem 1]{Feldman-Moore-2}, we have an isomorphism between countable measured equivalence relations $\cS_{(\alpha,\Phi)} \rtimes \cG$ and $\cR$.    
\end{proof}

The proof of Theorem \ref{Theorem C} now follows from Corollary \ref{corollary:strong normality in semi-direct product}, Construction \ref{Construction: quotient inverse semigroup and groupoid} and Theorem \ref{Thm: strong normailty correspondence with cocycle actions}. 

\printbibliography

\end{document}